\documentclass{article}
\usepackage{amsmath, amsthm, amssymb, graphicx, verbatim, graphicx, stmaryrd}
\usepackage{enumerate}
\usepackage{tikz-cd}
\usepackage[utf8]{inputenc}
\usepackage[margin=1in]{geometry}
\usepackage{hyperref}

\theoremstyle{plain}% default
\newtheorem{theorem}{Theorem}[section]
\newtheorem{corollary}[theorem]{Corollary}
\newtheorem{proposition}[theorem]{Proposition}
\newtheorem{lemma}[theorem]{Lemma}

\theoremstyle{definition}
\newtheorem{definition}[theorem]{Definition}
\newtheorem{remark}[theorem]{Remark}

\numberwithin{equation}{subsection}

\let\oldgamma\gamma
\renewcommand{\gamma}{\raisebox{\depth}{$\oldgamma$}}
\newcommand{\amod}{{}_{\mathcal{A}}^{\mathrm{pf}}\mathbf{Mod}}
\newcommand{\samod}{{}_{\mathcal{A}}^{\mathrm{ds}}\mathbf{Mod}}
\newcommand{\cfs}{\mathrm{cfs}}
\newcommand{\Supp}{\mathrm{Supp}}
\newcommand{\Blocks}{\mathrm{Blocks}}
\newcommand{\Fact}{\mathrm{Fact}}
\newcommand{\End}{\mathrm{End}}
\newcommand{\Hom}{\mathrm{Hom}}
\newcommand{\Ext}{\mathrm{Ext}^1_\Gamma}
\newcommand{\HC}{\mathrm{HC}}
\newcommand{\SHC}{\mathrm{SHC}}
\newcommand{\fwords}[1]{#1^{\Pi}}
\newcommand{\cfsclasses}{\mathrm{cfs}(\Gamma)/{\sim}}

\title{On the Category of Harish-Chandra Block Modules}
\author{Dylan Fillmore }
\date{\today}

\begin{document}

\maketitle

\begin{abstract}

If $\Gamma$ is a subalgebra of $A$, then an $A$-module is called a \emph{Harish-Chandra module} if it is the direct sum of its generalized weight spaces with respect to $\Gamma$. Drozd, Futorny, and Ovsienko \cite{drozd1994harish} defined a generalization of a central subalgebra called a \emph{Harish-Chandra subalgebra} and showed that when $\Gamma$ is a Harish-Chandra subalgebra of $A$ the structure of Harish-Chandra $A$-modules can be described using information about the relationship between $A$ and the cofinite maximal ideals of $\Gamma$.

We extend the results of \cite{drozd1994harish} by dropping the assumption that $\Gamma$ is quasicommutative. We facilitate this by introducing an equivalence relation $\sim$ on the set $\cfs(\Gamma)$ of cofinite maximal ideals of $\Gamma$. We define \emph{Harish-Chandra block modules} with respect to $\sim$ to be $A$-modules that are the direct sum of so called block spaces corresponding to the equivalence classes $\cfsclasses$. If $\Gamma$ is a \emph{Harish-Chandra block subalgebra} of $A$ with respect to $\sim$, then the structure of Harish-Chandra block modules can be described based on the relationship between $A$ and $\cfsclasses$. In particular, we give a decomposition of the category of Harish-Chandra block modules and the collection of isomorphism classes of irreducible Harish-Chandra block modules. Furthermore, we define a category $\mathcal{A}$ on $\cfsclasses$, and show the category of profinite $\mathcal{A}$-modules is equivalent to the category of Harish-Chandra block modules. Taking $\Gamma$ to be noetherian and quasicommutative, and $\sim$ to be the equality relation, we recover (in fact, a slight refinement of) results from \cite{drozd1994harish}. Lastly, we provide a sufficient condition for when there are a finite number of isoclasses of simple Harish-Chandra block modules with a given support.
\end{abstract}

\tableofcontents

\section{Introduction}
Given an associative algebra $A$, one can approach the study of $A$-modules by restricting to a subcategory of modules which are more manageable with respect to a subalgebra $\Gamma$. A classical example in representation theory is given by taking a Lie algebra $\frak{k}$ which is reductive in a finite-dimensional Lie algebra $\frak{g}$. Letting $A = U(\frak{g})$ and $\Gamma = U(\frak{k})$, a Harish-Chandra module \cite{dixmier1996enveloping} is an $A$-module $V$ which is the sum of its finite-dimensional simple $\Gamma$-submodules. Equivalently, if $\cfs(\Gamma)$ denotes the set of maximal ideals of $\Gamma$ with finite codimension, then $V$ is the sum of the $\Gamma$-weight spaces $V^\frak{m} = \{v \in V: \frak{m}v = 0 \}$ where $\frak{m}$ ranges across $\cfs(\Gamma)$.

When $\frak{g}$ is a reductive Lie algebra with Cartan subalgebra $\frak{h}$, Fernando \cite{fernando1990lie} and Mathieu \cite{mathieu2000classification} studied the $\frak{g}$-modules which decomposed into a direct sum of finite-dimensional weight spaces with respect to $\frak{h}$. They classified the irreducible weight modules of this form. This result has been reproduced in the context of Drinfeld-Jimbo quantum groups $U_q(\mathfrak{g})$ \cite{pedersen2015irreducible1} \cite{pedersen2015irreducible2}.

When $\Gamma$ is an algebra and $A$ is a twisted generalized Weyl algebra with base algebra $\Gamma$, the simple weight modules without inner breaks have been classified in \cite{hartwig2006locally} (see also \cite{hartwig2018noncommutative}).

Work by Lepowsky and McCollum \cite{lepowsky1973determination} encompasses the preceding examples with the framework of an arbitrary algebra $A$ and subalgebra $\Gamma$. If $V$ is an $A$-module, then they relate the action of $A$ on $V^\frak{m}$ to the $\Gamma$-module $A/A\frak{m}$. In particular, if $A/A\frak{m}$ is a $\Gamma$-weight module, then $A V^\frak{m} \subseteq \bigoplus_{\frak{n} \in \cfs(\Gamma)} V^\frak{n}$. Additionally, they give a correspondence between irreducible $A$-modules $V$ with $V^\frak{m} \neq 0$ and irreducible $\big(A/A\frak{m}\big) ^\frak{m}$-modules. (Here $\big(A/A\frak{m}\big) ^\frak{m}$ coincides with the quotient of the normalizer $N_A(A\frak{m})$ by $A\frak{m}$.)

Another motivating example is $A=U(\frak{gl}_n)$ and $\Gamma$ the Gelfand-Tsetlin subalgebra, by definition the subalgebra generated by the centers $Z(\frak{gl}_1), \ldots, Z(\frak{gl}_n)$. Then Gelfand-Tsetlin modules are defined as $\frak{gl}_n$-modules $V$ which are locally finite with respect to $\Gamma$ \cite{drozd1994harish}. Equivalently (since $\Gamma$ is noetherian), $V$ is the direct sum of its generalized $\Gamma$-weight spaces $V(\frak{m}) = \{v \in V: \frak{m}^k v = 0 \text{ for some } k \geq 0\}$. Gelfand-Tsetlin modules have been thoroughly studied over the past decades and generalized to the setting of Galois algebras (see \cite{futorny2018representations} and references therein). Webster \cite{webster2019gelfand} introduced the framework of principal flag orders and used it along with methods of geometric representation theory to give a classification of the irreducible Gelfand-Tsetlin modules.

Drozd, Futorny, and Ovsienko \cite{drozd1994harish} unified these examples under the name Harish-Chandra modules. A Harish-Chandra module is an $A$-module $V$ which is the direct sum of its generalized $\Gamma$-weight spaces. They required that $\Gamma$ is a so called Harish-Chandra subalgebra of $A$, and showed that the action of $a \in A$ on $V(\frak{m})$ depends on the composition factors of $\frac{\Gamma a \Gamma}{\Gamma a \frak{m}}$. Consequently, the category of Harish-Chandra modules and the collection of isomorphism classes of irreducible Harish-Chandra modules can be decomposed. Additionally, the authors of \cite{drozd1994harish} defined a category $\mathcal{A}$ with objects $\cfs(\Gamma)$ and morphisms $\mathcal{A}(\frak{m}, \frak{n}) = \varprojlim \frac{A}{\frak{n}^n A + A \frak{m}^m}$, and established an equivalence between the category of Harish-Chandra modules and the category of discrete $\mathcal{A}$-modules.

We drop the assumption of \cite{drozd1994harish} that $\Gamma$ is quasicommutative. That is, $\Ext$ may be nonzero between simple modules corresponding to distinct ideals in $\cfs(\Gamma)$. Consequently, finite-dimensional $\Gamma$-modules may no longer be the direct sum of their generalized weight spaces. In general, however, if $\sim$ is the strongest equivalence relation on $\cfs(\Gamma)$ such that $\Ext$ is zero between simple modules corresponding to ideals in distinct equivalence classes of $\sim$, then every finite-dimensional $\Gamma$-module $V$ is the direct sum of the block spaces $V(B) = \{v \in V: \frak{m}_1 \cdots \frak{m}_k v = 0 \text{ for some } \frak{m}_1, \ldots, \frak{m}_k \in B\}$ where $B$ ranges across $\cfsclasses$. (In particular, the category of finite-dimensional $\Gamma$-modules decomposes into blocks corresponding to the equivalence classes of $\sim$. See, for example, \cite{humphreys2008representations}). We call an $A$-module $V$ a \emph{Harish-Chandra block module} (with respect to an equivalence relation $\sim$) if $V = \bigoplus_{B \in \cfsclasses} V(B)$, and denote by $\Supp(V)$ the set $\{B \in \cfsclasses: V(B) \neq 0\}$. We denote the full subcategory of Harish-Chandra block modules by $\HC(A; \Gamma, \sim)$ and the collection of isoclasses of irreducible Harish-Chandra block modules by $\mathrm{Irr}(A; \Gamma, \sim)$.

We extend the work of \cite{drozd1994harish} to Harish-Chandra block modules, under the assumption that $\Gamma$ is a \emph{Harish-Chandra block subalgebra} of $A$ (a modification of the notion of Harish-Chandra subalgebra from \cite{drozd1994harish}, see Section \ref{subsec:dfosubalgebras} for details). In this situation, if $\frak{m}_1, \ldots, \frak{m}_k \in B$ and $\frak{m}_1 \cdots \frak{m}_k v = 0$, then $Av \subseteq \bigoplus_{C \in \Supp(A/A\frak{m}_1 \cdots \frak{m}_k)} V(C)$. This suggests a preorder $\prec$ on the blocks of $\cfs(\Gamma)$. We define the equivalence relation $\Delta$ generated by $\prec$ and the equivalence relation $\nabla$ induced by $\prec$. If $\mathcal{D} \subseteq \cfsclasses$, then we let $\HC(A; \Gamma, \sim; \mathcal{D})$ denote the subcategory of Harish-Chandra block modules $V$ with $\Supp \subseteq \mathcal{D}$ and $\mathrm{Irr}(A; \Gamma, \sim; \mathcal{D})$ denote the collection of isoclasses of such simple Harish-Chandra block modules. We prove the following theorem in Section \ref{sec:decomp}, analogous to Corollary 16 of \cite{drozd1994harish}.

\begin{theorem}
Suppose $\Gamma$ is a Harish-Chandra block subalgebra of $A$.
\begin{enumerate}[{\rm (i) }]
    \item $\HC(A; \Gamma, \sim) = \bigoplus_{\mathcal{D} \in (\cfsclasses)/\Delta} \HC(A; \Gamma, \sim; \mathcal{D})$
    \item $\mathrm{Irr}(A; \Gamma, \sim) = \bigsqcup_{\mathcal{D} \in (\cfsclasses)/\nabla} \mathrm{Irr}(A; \Gamma, \sim; \mathcal{D})$
\end{enumerate}
\end{theorem}

If for each $B \in \cfsclasses$, there are $\frak{m}_1, \ldots, \frak{m}_k \in B$ so that $\frak{m}_1 \cdots \frak{m}_k V(B) = 0$, then we call $V$ a \emph{strong Harish-Chandra block module}. If we assume that $\Gamma$ is a \emph{strong Harish-Chandra block subalgebra} of $A$ (see Section \ref{sec:strong} for definition), then we may define a category $\mathcal{A}$ with objects $\cfsclasses$ and morphisms $\mathcal{A}(B,C) = \varprojlim \frac{A}{\frak{n}_1 \ldots \frak{n}_n A + A \frak{m}_1 \ldots \frak{m}_m}$. We define the categories $\amod$ and $\samod$ of profinite and discrete $\mathcal{A}$-modules, respectively. In Section \ref{sec:category}, we prove the following two equivalences:

\begin{theorem}\label{thm:main2}
If $\Gamma$ is a strong Harish-Chandra block subalgebra of $A$ with respect to $\sim$, then:
\begin{enumerate}[{\rm (i) }]
    \item $\amod \cong \HC(A; \Gamma, \sim)$ \label{thm:main2i}
    \item $\samod \cong \SHC(A; \Gamma, \sim)$ \label{thm:main2ii}
\end{enumerate}
\end{theorem}

As a special case of our work, we recover the corresponding results concerning the Harish-Chandra modules of \cite{drozd1994harish}. Additionally, our notion of strong Harish-Chandra block modules allows us to clarify ambiguities of \cite{drozd1994harish} raised by the possibility of infinite-dimensional generalized weight spaces. In particular, through Theorem \ref{thm:main2}(\ref{thm:main2ii}) we amend the category equivalence given in \cite{drozd1994harish} by replacing the category of Harish-Chandra modules with the category of strong Harish-Chandra modules. Theorem \ref{thm:main2}(\ref{thm:main2i}) instead states that the category of Harish-Chandra modules (with possibly infinite-dimensional generalized weight spaces) is equivalent to the category of profinite $\mathcal{A}$-modules.

In Section \ref{sec:irr}, we prove the following theorem, analogous to Corollary 19 of \cite{drozd1994harish}.

\begin{theorem}\label{thm:irreducibles}
    Let $\Gamma$ be a strong Harish-Chandra block subalgebra of $A$ with respect to $\sim$, and let $B \in \cfsclasses$. Suppose $B$ is finite, $\Gamma$ is noetherian, and $\mathcal{A}(B,B)$ is finitely generated as a left and right $\hat{\Gamma}_B$-module. Then:
    \begin{enumerate}[{\rm (i) }]
        \item There are finitely many isoclasses of simple Harish-Chandra block modules $V$ with $B \in \Supp(V)$.
        \item If $V$ is a simple Harish-Chandra block module $V$, then $V(B)$ is finite-dimensional.
    \end{enumerate}
\end{theorem}

Throughout, we illustrate with an upper triangular algebra of two copies of $U(\frak{sl}_2)$ glued along a one-dimensional bimodule.

\section{Preliminaries}

Let $\Bbbk$ be a field and $\Gamma$ a $\Bbbk$-algebra. We denote by $\cfs(\Gamma)$ the set of maximal two-sided ideals $\frak{m}$ of $\Gamma$ such that $\Gamma / \frak{m}$ is finite-dimensional. For $\frak{m} \in \cfs(\Gamma)$, $\Gamma / \frak{m}$ is isomorphic to a matrix algebra over a finite-dimensional division algebra, by the Artin-Wedderburn Theorem. In particular, there is a unique simple $\Gamma / \frak{m}$-module denoted $S_\frak{m}$. Given a finite-dimensional $\Gamma$-module $V$, we denote by $\Fact(V)$, the set of composition factors of $V$. By the following lemma, $\Fact(V)$ corresponds to a subset of $\cfs(\Gamma)$.

\begin{lemma}\label{Jacobson}
If $V$ is a finite-dimensional $\Gamma$-module, then $\Fact(V) \subseteq \{S_\frak{m}: \frak{m} \in \cfs(\Gamma)\}$.
\end{lemma}

\begin{proof}
If $S$ is a simple finite-dimensional $\Gamma$-module, then the Jacobson Density Theorem implies that $\Gamma \rightarrow \End_D(S)$ is surjective (where $D = \End_\Gamma(S)$). So $\Gamma / \mathrm{Ann}(S)$ is isomorphic to the simple $\Bbbk$-algebra $\End_D(S)$ which is finite-dimensional since it is a subalgebra of the finite-dimensional algebra $\End_\Bbbk(S)$. Consequently, $\mathrm{Ann}(S) \in \cfs(\Gamma)$ with $\mathrm{Ann}(S)S = 0$. So $S$ is isomorphic to the unique simple $\Gamma / \mathrm{Ann}(S)$-module.

Now, of course, any composition factor of $V$ is a simple finite-dimensional $\Gamma$-module.
\end{proof}

\subsection{Ext}
In this section, we review some facts about $\mathrm{Ext}$ groups.

Let $U, W$ be $\Gamma$-modules. Recall that $\Ext(U, W) = 0$ if and only if every short exact sequence $0 \rightarrow W \rightarrow V \rightarrow U \rightarrow 0$ of $\Gamma$-modules splits.

\begin{lemma} \label{remainderlemma}
    Let $\frak{m}, \frak{n} \in \cfs(\Gamma)$ be distinct ideals and let $W$ be a $\Gamma / \frak{m}$-module and $U$ be a $\Gamma / \frak{n}$-module. If $\frak{m}\frak{n} = \frak{m} \cap \frak{n}$, then $\Ext(U, W) = 0$.
\end{lemma}

\begin{proof}
    Let $0 \rightarrow W \xrightarrow{f} V \xrightarrow{g} U \rightarrow 0$ be a short exact sequence. Then $g(\frak{n}V) = \frak{n}g(V) = 0$, so $\frak{n}V \subseteq f(W)$. Hence $\frak{m}\frak{n}V \subseteq \frak{m}f(W) = f(\frak{m}W) = 0$. So $V$ is a $\Gamma / \frak{m}\frak{n}$-module. By assumption $\Gamma / \frak{m} \frak{n} = \Gamma / \frak{m} \cap \frak{n}$ which is isomorphic to $\Gamma/ \frak{m} \times \Gamma / \frak{n}$ by the remainder theorem. So $V = V_1 \oplus V_2$ where $\frak{m} V_1 = 0$ and $\frak{n} V_2 = 0$.

    Now $V_1 = \frak{m}V_1 + \frak{n}V_1 = \frak{n}V_1$, so $V_1 \subseteq \ker(g) = f(W)$. Similarly, $W = \frak{n}W$ gives $\pi_2(f(W)) = \frak{n}\pi_2(f(W)) = 0$. So $V_1 = f(W)$. The sequence is split by the retraction $h^{-1} \circ \pi_1$ where $h: W \rightarrow V_1$ such that $f = \iota_1 \circ h$.
\end{proof}

The following result is well-known, but we provide a proof for convenience.

\begin{lemma}\label{extlemma}
Let $M, N$ be finite-dimensional $\Gamma$-modules with $\Ext(S , T) = 0$ for all $S \in \Fact(M)$ and $T \in \Fact(N)$. Then $\Ext(M, N) = 0$.
\end{lemma}

\begin{proof}
We prove the lemma by induction on the lengths of $M$ and $N$. Suppose first that $M$ is simple. Let $N_1$ be a simple submodule of $N$. The short exact sequence $0 \rightarrow N_1 \rightarrow N \rightarrow N/N_1 \rightarrow 0$ yields a long exact sequence. In particular, we have that $\Ext(M, N_1) \rightarrow \Ext(M,N) \rightarrow \Ext(M, N/N_1)$ is exact. By assumption, $\Ext(M, N_1) = 0$. By induction, $\Ext(M, N/N_1) = 0$. Hence, $\Ext(M,N) = 0$.

Now if $M$ is not simple, let $M_1$ be a simple submodule of $M$. The short exact sequence $0 \rightarrow M_1 \rightarrow M \rightarrow M/M_1 \rightarrow 0$ yields a long exact sequence. In particular, we have that $\Ext(M/M_1, N) \rightarrow \Ext(M, N) \rightarrow \Ext(M_1, N)$ is exact. By the previous paragraph, $\Ext(M_1, N) = 0$.  By induction, $\Ext(M/M_1, N) = 0$. Hence $\Ext(M,N) = 0$.
\end{proof}

\subsection{Block Modules}
In this section, we extend the notion of generalized weight modules to a situation which allows for a not necessarily quasicommutative $\Gamma$.

Let $\sim$ be an equivalence relation on $\cfs(\Gamma)$. We denote the set of equivalence classes of $\sim$ by $\cfsclasses$ and the equivalence class of $\frak{m} \in \cfs(\Gamma)$ by $[\frak{m}]_\sim$. It is sufficient for our purposes to take $\sim$ to be an arbitrary equivalence relation on $\cfs(\Gamma)$, but Proposition \ref{prop:extmodule} shows it is natural to take $\sim$ to be the equivalence relation generated by $\frak{m} \sim \frak{n}$ when $\Ext(S_\frak{m}, S_\frak{n}) \neq 0$. We will refer to this relation as the $\mathrm{Ext}$ relation.

Given $B \in \cfsclasses$, let \begin{equation}\fwords{B} = \{\frak{m}_1 \ldots \frak{m}_k: k \geq 0 \text{ and } \frak{m}_1, \ldots, \frak{m}_k \in B\}.\end{equation} If $V$ is a $\Gamma$-module and $B \in \cfsclasses$, then the corresponding \emph{block space} is the $\Gamma$-submodule defined by \begin{equation}V(B) = \{v \in V: \frak{m} v = 0 \text{ for some } \frak{m} \in \fwords{B}\}.\end{equation} In fact, the sum $\bigoplus_{B \in \cfsclasses} V(B)$ is direct. We define the \emph{support} of $V$ to be the set \begin{equation}\Supp(V) = \{B \in \cfsclasses : V(B) \neq 0\}.\end{equation}

\begin{definition}
Let $V$ be a $\Gamma$-module. If $V = \bigoplus_{B \in \cfsclasses} V(B)$, then we call $V$ a \emph{block module} (with respect to $\sim$). We denote the projections and inclusions by $\pi_B$ and $\iota_B$, respectively. Let $\mathrm{Bl}(\Gamma, \sim)$ be the full subcategory of all block modules.
\end{definition}

We will discuss maps between Harish-Chandra block modules later, so for now we make note of the following proposition.

\begin{proposition}\label{prop:blockmaps}
If $V$ is a block module, any $\Gamma$-module map $f: V \rightarrow W$ maps into the submodule $\bigoplus_{B \in \Supp(V)} W(B)$. Furthermore, if $W$ is also a block module, then $f: V \rightarrow W$ is of the form $\bigoplus_{B \in \cfsclasses} f_B$ where $f_B: V(B) \rightarrow W(B)$ is $\pi_B \circ f \circ \iota_B$.
\end{proposition}

\begin{proof}
    Suppose $V$ is a block module, and $f:V \rightarrow W$ is a $\Gamma$-module map. If $v \in V(B)$, then there is $\frak{m} \in \fwords{B}$ such that $\frak{m} v = 0$. So $\frak{m} f(v) = f(\frak{m} v) = f(0) = 0$. So $f(v) \in W(B)$. The rest is clear.
\end{proof}

Throughout this paper, we will often need to refer to the following proposition, which shows that submodules and quotients of block modules are also block modules

%We will see later that $V, W$ can be made into functors, such that the HC module maps are precisely the natural transformations between such functors.

\begin{proposition}\label{SES}
If $0 \rightarrow U \rightarrow V \rightarrow W \rightarrow 0$ is a short exact sequence of $\Gamma$-modules, and $V$ is a block module, then $U$ and $W$ are block modules with $\Supp(V) = \Supp(U) \cup \Supp(W)$.
\end{proposition}

\begin{proof}
Let $0 \rightarrow U \rightarrow V \rightarrow W \rightarrow 0$ be a short exact sequence of $\Gamma$-modules, where $V$ is a block module. The surjectivity of $V \twoheadrightarrow W$ implies that $W$ is a block module with $\Supp(W) \subseteq \Supp(V)$, by Proposition \ref{prop:blockmaps}.

Let $u \in U$ (we assume that $U$ is a submodule of $V$). Then $u = u_1 + \ldots + u_k$ with $u_i \in V(B_i)$ for some $B_i \in \cfsclasses$. Let $\frak{m}_i \in \fwords{B_i}$ such that $\frak{m}_i u_i = 0$. Now given $j \in [k]$, $\frak{m}_j + \prod_{i \neq j} \frak{m}_i = \Gamma$, so there is $\gamma_j \in \frak{m}_j$ and $\mu_j \in \prod_{i \neq j} \frak{m}_i$ so that $\gamma_j + \mu_j = 1$. So $u_j = \mu_j u_j = \mu_j (u - \sum_{i \neq j} u_i) = \mu_j u \in U$. So each $u_i \in U$, and $U$ is a block module with $\Supp(U) \subseteq \Supp(V)$.
\end{proof}

We now have the tools required to prove Proposition \ref{prop:extmodule}, which indicates that a natural choice of $\sim$ is the $\mathrm{Ext}$ relation.

\begin{proposition}\label{prop:extmodule}
    Let $\sim$ be the equivalence relation generated by $\frak{m} \sim \frak{n}$ when $\Ext(S_\frak{m}, S_\frak{n}) \neq 0$. Let $V$ be a finite-dimensional $\Gamma$-module. Then $V$ is a block module with respect to $\sim$.
\end{proposition}

We will first prove the following lemma.

\begin{lemma}\label{blockfactors}
If $V$ is a finite-dimensional block module, then $\Fact(V) \subseteq \bigcup_{B \in \Supp(V)} \{S_\frak{m}: \frak{m} \in B\}$.
\end{lemma}

\begin{proof}
We have $V = \bigoplus_{i=1}^k V(B_i)$ for some $B_i \in \cfsclasses$. So $\Fact(V) = \bigcup_{i=1}^k \Fact(V(B_i))$. Since $\Supp(V) = \{B_1,...,B_k\}$, it is sufficient to show $\Fact(V(B_i)) \subseteq \{S_\frak{m}: \frak{m} \in B_i\}$.

Given $\frak{m} \in \fwords{B_i}$, let $V^\frak{m} = \{v \in V: \frak{m} v = 0 \}$. Since $V$ is finite-dimensional, there are $\frak{m}_1, \ldots, \frak{m}_m \in B_i$ such that $V(B_i) = V^{\frak{m}_1 \ldots \frak{m}_m}$. Then $0 \subseteq V^{\frak{m}_1} \subseteq \cdots \subseteq V^{\frak{m}_1 \ldots \frak{m}_m} = V(B_i)$ is a sequence of $\Gamma$-modules, which can be refined to a composition series, where each composition factor is a simple finite-dimensional $\Gamma / \frak{m}_j$-module for some $j$.
\end{proof}

\begin{proof}[Proof of Proposition \ref{prop:extmodule}]
We show that every finite-dimensional $\Gamma$-module $V$ is a block module, by induction on the dimension of $V$. If $V = 0$, then $V$ is trivially a block module. Otherwise, for any $B \in \Supp(V)$, $W := V / V(B)$ is a block module (with $W(B) = 0$), by induction. So it will be sufficient to show that $\Ext(W, V(B))$ is zero. Note \begin{equation}\Ext(W, V(B)) = \Ext(\bigoplus_{C \neq B} W(C), V(B)) = \prod_{C \neq B} \Ext(W(C), V(B)).\end{equation} Lemma \ref{blockfactors} implies that $\Fact(W(C)) \subseteq \{S_\frak{m}: \frak{m} \in C\}$ and $\Fact(V(B)) \subseteq \{S_\frak{m}: \frak{m} \in B\}$. Since $B,C$ are distinct blocks, $\Ext(W(C), V(B)) = 0$ by Lemma \ref{extlemma}.
\end{proof}

\subsection{Examples}
\subsubsection{Weight Modules}
Here we show how block modules extend the idea of generalized weight modules.

\begin{definition}
    If $\sim$ is the equality relation, then a block module is called a \emph{generalized weight module}.
\end{definition}

In \cite{drozd1994harish}, the following generalization of a commutative algebra was defined, which is called a quasicommutative algebra. Generalized weight modules were explored in \cite{drozd1994harish} under the assumption that $\Gamma$ is quasicommutative.

\begin{definition}[\cite{drozd1994harish}]
We say $\Gamma$ is \emph{quasicommutative} if for all distinct $\frak{m}, \frak{n} \in \cfs(\Gamma)$, $\Ext(S_{\frak{m}}, S_{\frak{n}}) = 0$.
\end{definition}

In other words, $\Gamma$ is quasicommutative if the $\mathrm{Ext}$ relation is the equality relation. We provide an alternative characterization of quasicommutative algebras that we found more informative.

\begin{proposition}
$\Gamma$ is quasicommutative if and only if every finite-dimensional $\Gamma$-module is a generalized weight module.
\end{proposition}

\begin{proof}
Suppose every finite-dimensional $\Gamma$-module is a generalized weight module. Let $\frak{m}, \frak{n} \in \cfs(\Gamma)$ be distinct and let $0 \rightarrow S_{\frak{n}} \xrightarrow{f} V \xrightarrow{g} S_{\frak{m}} \rightarrow 0$ be a short exact sequence of $\Gamma$-modules. Then $V$ is finite-dimensional, so by assumption $V$ is a generalized weight module. So this is a short exact sequence of generalized weight modules. In particular, we have:
\[
\begin{tikzcd}
S_\frak{n} \arrow[r, "f"] & V \arrow[r, "g"] & S_\frak{m} \\
S_\frak{n} \arrow[r, "f_\frak{n}"] \arrow[u, equals] & V(\frak{n}) \arrow[r, "g_\frak{n}"] \arrow[u, "\iota_\frak{n}"] & 0 \arrow[u]
\end{tikzcd}
\]
so $f_\frak{n}$ is surjective, since $V(\frak{n}) \subseteq \ker(g)$. So $f$ splits via the map $f_\frak{n}^{-1} \circ \pi_\frak{n}$. Thus $\Ext(S_{\frak{m}}, S_{\frak{n}}) = 0$ for all distinct $\frak{m}, \frak{n} \in \cfs(\Gamma)$.

The forward direction is given by Proposition \ref{prop:extmodule}.
\end{proof}

\subsubsection{Upper Triangular Algebras}
Here we construct a block module which is not a generalized weight module. Let $\Gamma_0 = \mathbb{C}[h, c]$. In Section \ref{subsec:hctri}, $\Gamma_0$ will play the role of the Gelfand-Tsetlin subalgebra of $U(\frak{sl}_2)$. Let $\frak{z}$ denote the ideal $(h,c) \in \cfs(\Gamma_0)$. Take $\mathbb{C}x$ to be the one dimensional $(\Gamma_0,\Gamma_0)$-bimodule where $h,c$ act by $0$. Then $$\Gamma = \begin{bmatrix} \Gamma_0 & \mathbb{C}x \\ 0 & \Gamma_0 \end{bmatrix}$$ is an algebra under addition and multiplication of matrices.
Given a two-sided ideal $I$ of $\Gamma_0$, define \begin{equation}\overline{I} := \begin{bmatrix} I & \mathbb{C}x \\ 0 & \Gamma_0 \end{bmatrix} \qquad \text{ and } \qquad \underline{I} := \begin{bmatrix} \Gamma_0 & \mathbb{C}x \\ 0 & I \end{bmatrix}.\end{equation} Then every ideal of $\cfs(\Gamma)$ is of the form $\overline{\frak{m}}$ or $\underline{\frak{m}}$ for some $\frak{m} \in \cfs(\Gamma_0)$.

We will take $\sim$ to be the $\mathrm{Ext}$ relation, that is, the equivalence relation on $\cfs(\Gamma)$ generated by $\frak{m} \sim \frak{n}$ when $\Ext(S_\frak{m}, S_\frak{n}) \neq 0$.
Note that for any two-sided ideals $I,J$ of $\Gamma_0$, we have $\overline{I}\ \overline{J} = \overline{IJ} = \overline{I \cap J} = \overline{I} \cap \overline{J}$ and $\underline{I}\ \underline{J} = \underline{IJ} = \underline{I \cap J} = \underline{I} \cap \underline{J}$. So by Lemma \ref{remainderlemma}, $\Ext(S_{\overline{\frak{m}}}, S_{\overline{\frak{n}}}) = 0$ and $\Ext(S_{\underline{\frak{m}}}, S_{\underline{\frak{n}}}) = 0$ for $\frak{m}, \frak{n} \in \cfs(\Gamma)$. Moreover, for any two-sided ideals $I,J$ of $\Gamma_0$, we have $\underline{I}\overline{J} = \underline{I} \cap \overline{J}$, so again $\Ext(S_{\overline{\frak{m}}}, S_{\underline{\frak{n}}}) = 0$ for $\frak{m}, \frak{n} \in \cfs(\Gamma)$. Lastly, for any two-sided ideals $I,J$ of $\Gamma_0$, we have $$\overline{I}\underline{J} = \begin{bmatrix} I & I.\mathbb{C}x + \mathbb{C}x.J \\ 0 & J \end{bmatrix}.$$ Since $\mathbb{C}x$ is simple as both a left and right $\Gamma_0$-module, when $\frak{m}, \frak{n} \in \cfs(\Gamma_0)$, we have $\frak{n}.\mathbb{C}x + \mathbb{C}x.\frak{m} = 0$ if $\frak{m} = \frak{n} = \frak{z}$ and $\frak{n}.\mathbb{C}x + \mathbb{C}x.\frak{m} = \mathbb{C}x$ otherwise. Consequently, if $\frak{m}, \frak{n} \in \cfs(\Gamma_0)$ are not both equal to $\frak{z}$, then $\overline{\frak{n}}\underline{\frak{m}} = \overline{\frak{n}} \cap \underline{\frak{m}}$ and $\Ext(S_{\underline{\frak{m}}}, S_{\overline{\frak{n}}}) = 0$. We next demonstrate that $\Ext(S_{\underline{\frak{z}}}, S_{\overline{\frak{z}}}) \neq 0$, which shows that \begin{equation}\cfsclasses = \{ \{\overline{\frak{z}}, \underline{\frak{z}}\} \} \cup \{ \{\overline{\frak{m}}\}\}_{\frak{m} \in \cfs(\Gamma_0)\setminus \{\frak{z}\}} \cup \{\{\underline{\frak{m}}\} \}_{\frak{m} \in \cfs(\Gamma_0)\setminus \{\frak{z}\}}.\end{equation} Geometrically, $\cfs(\Gamma_0)$ is $\mathbb{C}^2$ and $\cfsclasses$ is two copies of $\mathbb{C}^2$ with their origins identified.

Since $\Gamma$ acts on its own columns, we have the $\Gamma$-module $\begin{bmatrix} \mathbb{C}x \\ \Gamma_0 \end{bmatrix}$ with submodule $\begin{bmatrix} 0 \\ \frak{z} \end{bmatrix}$. Consider the quotient module $\begin{bmatrix} \mathbb{C}x \\ \Gamma_0/\frak{z} \end{bmatrix}$. The short exact sequence of $\Gamma$-modules $$0 \rightarrow \mathbb{C}x \rightarrow \begin{bmatrix} \mathbb{C}x \\ \Gamma_0/\frak{z} \end{bmatrix} \rightarrow \Gamma_0 / \frak{z} \rightarrow 0$$ does not split, as $\begin{bmatrix} \mathbb{C}x \\ 0 \end{bmatrix}$ is the only nonzero proper submodule of $\begin{bmatrix} \mathbb{C}x \\ \Gamma_0 / \frak{z} \end{bmatrix}$.
Now $\mathbb{C}x$ is the simple $1$-dimensional $\Gamma / \overline{\frak{z}}$-module, and $\Gamma_0 / \frak{z}$ is the simple $1$-dimensional $\Gamma / \underline{\frak{z}}$-module. So $\Ext(S_{\underline{\frak{z}}}, S_{\overline{\frak{z}}}) \neq 0$.

This upper triangular algebra $\Gamma$ demonstrates that the $\mathrm{Ext}$ relation may be nontrivial. Furthermore, the $\Gamma$-module $\begin{bmatrix} \mathbb{C}x \\ \Gamma_0/\frak{z} \end{bmatrix}$ is a block module with respect to $\mathrm{Ext}$, but not a generalized weight module.

\section{Decomposition of $\HC(A; \Gamma, \sim)$}\label{sec:decomp}
In this section, we prove the first main theorem of this paper.
\subsection{Harish-Chandra Block Modules}

Let $A$ be a $\Bbbk$-algebra containing $\Gamma$. The authors of \cite{drozd1994harish} defined a Harish-Chandra module as an $A$-module which is a generalized weight module under the action of $\Gamma$. We extend this notion to block modules with respect to $\sim$.

\begin{definition}
If $V$ is an $A$-module which is a block module with respect to $\Gamma$ and $\sim$, we call $V$ a \emph{Harish-Chandra block module} (with respect to $\sim$). Let $\HC(A; \Gamma, \sim)$ be the full subcategory of $_A\mathbf{Mod}$ of all Harish-Chandra block modules with respect to $\sim$.
\end{definition}

If $V,W \in \HC(A; \Gamma, \sim)$ and $f: V \rightarrow W$ is a map of Harish-Chandra block modules, then we saw in Proposition \ref{prop:blockmaps} that $f$ splits into a direct sum of $\Gamma$-module (and thus linear) maps $f_B: V(B) \rightarrow W(B)$. The following proposition characterizes maps of Harish-Chandra block modules, and foreshadows the equivalence of these maps to natural transformations between certain functors.

\begin{proposition}\label{naturaltransformation}
Let $V, W \in \HC(A; \Gamma, \sim)$. Then $f: V \rightarrow W$ is a map of Harish-Chandra block modules if and only if it splits into a direct sum of linear maps $f_B: V(B) \rightarrow W(B)$ such that for all $a \in A$ and $B, C \in \cfsclasses$, the following diagram commutes:
\[
\begin{tikzcd}
V(B) \arrow[r, "f_B"] \arrow[d, "\pi_C(a.-)" left]& W(B) \arrow[d, "\pi_C(a.-)"]\\
V(C) \arrow[r, "f_C"] & W(C)
\end{tikzcd}
\]
where $\pi_B$ is the appropriate projection onto the block space corresponding to $B$.
\end{proposition}

\begin{proof}
Let $v \in V(B)$. Note that $a.v = \sum_{C \in \cfsclasses} \pi_C(a.v)$. Now let $f$ be a direct sum of linear maps $f_B: V(B) \rightarrow W(B)$. Then $f(a.v) = f(\sum_C \pi_C(a.v)) = \sum_C f_C(\pi_C(a.v))$ and $a.f(v) = a.f_B(v) = \sum_C \pi_C(a.f_B(v))$. So $f$ is a Harish-Chandra block module map if and only if we have $f_C(\pi_C(a.v)) = \pi_C(a.f_B(v))$ for all $B,C \in \cfsclasses$ and $a \in A$.
\end{proof}

\subsection{Action of $A$ on a Block Module}
We will discuss Harish-Chandra block modules under the assumption that $\Gamma$ is a Harish-Chandra block subalgebra of $A$ with respect to $\sim$.

\begin{definition}
If $A/A\frak{m}$ is a Harish-Chandra block module with respect to $\sim$ for all $B \in \cfsclasses$ and $\frak{m} \in \fwords{B}$, then we say $\Gamma$ is a \emph{Harish-Chandra block subalgebra} of $A$ (with respect to $\sim$).
\end{definition}

The reader should note that a Harish-Chandra block module with respect to $=$ is the same as a Harish-Chandra module in the sense of \cite{drozd1994harish}, but a Harish-Chandra block subalgebra with respect to $=$ is not the same as a Harish-Chandra subalgebra in the sense of \cite{drozd1994harish}. Instead, the definition of a Harish-Chandra block subalgebra with respect to $\sim$ may be compared to the assumption of \cite[\S~4]{lepowsky1973determination}.

The definition of a Harish-Chandra block subalgebra with respect to $\sim$ is motivated by the following proposition, which is reminiscent of %\cite[Prop.~4.2]{lepowsky1973determination} and
\cite[Prop.~14]{drozd1994harish}.

%Proposition 4.2 of \cite{lepowsky1973determination} and Proposition 14 of \cite{drozd1994harish}.

\begin{proposition} \label{prop:blockspan}
Let $\Gamma$ be a Harish-Chandra block subalgebra of $A$ with respect to $\sim$. If $V$ is an $A$-module and $v \in V$ such that $\frak{m}_1\cdots\frak{m}_k v = 0$ where $B \in \cfsclasses$ and $\frak{m}_1, \ldots \frak{m}_k \in B$, then \begin{equation}Av \subseteq \bigoplus_{C \in \bigcup_i \Supp(A/A\frak{m}_i)} V(C).\end{equation}
\end{proposition} 

\begin{proof}
Consider the surjective map of $\Gamma$-modules $f: A \otimes_\Gamma \Gamma v \rightarrow A v$ given by $a \otimes \gamma v \mapsto a\gamma v$. We will first show that $A \otimes_\Gamma \Gamma v$ is a block module with support contained in $\bigcup_i \Supp(A/A\frak{m}_i)$. Consequently, by Proposition \ref{SES}, $A v$ is a block module with $\Supp(Av) \subseteq \Supp(A \otimes_\Gamma \Gamma v) \subseteq \bigcup_i \Supp(A/A\frak{m}_i)$.

If $M$ is a $\Gamma$-module, $B \in \cfsclasses$, and $\frak{m}_1,\ldots,\frak{m}_k \in B$ such that $\frak{m}_1 \cdots \frak{m}_k M = 0$, then \begin{equation}\label{eq:tensor} A \otimes_\Gamma M \cong \big(A/A\frak{m}_1 \cdots \frak{m}_k\big) \otimes_\Gamma M \cong \bigoplus_{C \in \cfsclasses} \Big(\big(A/A\frak{m}_1 \cdots \frak{m}_k\big)(C) \otimes_\Gamma M\Big)\end{equation} is a block module since $\Gamma$ is a Harish-Chandra block subalgebra of $A$. We show that $\Supp(A \otimes_\Gamma M) \subseteq \bigcup_i \Supp(A/A\frak{m}_i)$ by induction on $k$.

If $k = 1$, then $\Supp(A \otimes_\Gamma M) \subseteq \Supp(A/A\frak{m}_1)$ is clear by \ref{eq:tensor}. Otherwise, consider the short exact sequence: $$0 \rightarrow \frak{m}_kM \rightarrow M \rightarrow M/\frak{m}_kM \rightarrow 0.$$ By tensoring with $A$ we have the exact sequence: $$A \otimes_\Gamma \frak{m}_kM \xrightarrow{f} A \otimes_\Gamma M \rightarrow A \otimes_\Gamma (M/\frak{m}_kM) \rightarrow 0,$$ which gives us the short exact sequence: \begin{equation}\label{eq:tensor2}0 \rightarrow f(A \otimes_\Gamma \frak{m}_kM) \hookrightarrow A \otimes_\Gamma M \rightarrow A \otimes_\Gamma (M/\frak{m}_kM) \rightarrow 0.\end{equation}
By induction, $A \otimes_\Gamma \frak{m}_kM$ and $A\otimes_\Gamma (M/\frak{m}_kM)$ are block modules with $\Supp(A \otimes_\Gamma \frak{m}_kM) \subseteq \bigcup_{i=1}^{k-1} \Supp(A/A\frak{m}_i)$ and $\Supp(A\otimes_\Gamma (M/\frak{m}_kM)) \subseteq \Supp(A/A\frak{m}_k)$. Furthermore, by Proposition \ref{SES}, the image $f(A\otimes_\Gamma \frak{m}_kM)$ is a block module with $\Supp(f(A\otimes_\Gamma \frak{m}_kM)) \subseteq \Supp(A\otimes_\Gamma \frak{m}_kM)$. Now applying Proposition \ref{SES} to \ref{eq:tensor2}, we have \begin{align*}
    \Supp(A \otimes_\Gamma M) &= \Supp(f(A\otimes_\Gamma \frak{m}_kM)) \cup \Supp(A\otimes_\Gamma (M/\frak{m}_kM))\\
    &\subseteq \Supp(A\otimes_\Gamma \frak{m}_kM) \cup \Supp(A\otimes_\Gamma (M/\frak{m}_kM))\\
    &\subseteq \bigcup_i \Supp(A/A\frak{m}_i)
\end{align*}
as desired.
\end{proof}

\subsection{Relations and Decompositions}

In this subsection, we generalize \cite[Cor.~15]{drozd1994harish} and \cite[Cor.~16]{drozd1994harish} for Harish-Chandra block modules with respect to $\sim$.

If $\mathcal{D}\subseteq \cfsclasses$, and $V$ is a Harish-Chandra block module, then let $V(\mathcal{D})$ denote the $\Gamma$-submodule $\bigoplus_{B \in \mathcal{D}} V(B)$. Let $\prec$ be the preorder generated by $\{(B, C): C \in \Supp(A/A\frak{m})\}$ for $B \in \cfsclasses$ and $\frak{m} \in B$. Then rephrasing Theorem \ref{prop:blockspan}, we can say:

\begin{theorem}\label{relationtheorem1}
Suppose $\Gamma$ is a Harish-Chandra block subalgebra of $A$ with respect to $\sim$ and $V$ is an $A$-module. Let $\mathcal{D} \subseteq \cfsclasses$. Then \begin{equation}A.\big(V(\mathcal{D})\big) \subseteq \bigoplus_{C \succ \mathcal{D}} V(C).\end{equation}
\end{theorem}

Let $\Delta$ be the equivalence relation generated by $\prec$ (that is, the weakly connected components of $(\cfsclasses, \prec)$). Let $\nabla$ be the equivalence relation induced by $\prec$ (that is, the strongly connected components of $(\cfsclasses, \prec)$). Then Theorem \ref{relationtheorem1} has the following corollaries for $\HC(A; \Gamma, \sim)$.

\begin{corollary}
Let $\Gamma$ be a Harish-Chandra block subalgebra of $A$. Let $V$ be a Harish-Chandra block module.
\begin{enumerate}[{\rm (i) }]
    \item If $\mathcal{D} \subseteq \cfsclasses$ is $\prec$ closed, then $V(\mathcal{D})$ is an $A$-submodule of $V$.
    \item $V = \bigoplus_{\mathcal{D} \in \big(\cfsclasses\big)/\Delta} V(\mathcal{D})$ as $A$-modules.
    \item If $V$ is indecomposable, and $B \in \Supp(V)$, then $\Supp(V) \subseteq \Delta B$.
    \item If $V$ is irreducible, and $B \in \Supp(V)$, then $\Supp(V) \subseteq \nabla B$.
\end{enumerate}
\end{corollary}

Defining $\HC(A; \Gamma, \sim; \mathcal{D})$ to be the full subcategory of $\HC(A; \Gamma, \sim)$ of Harish-Chandra block modules $V$ with $\Supp(V) \subseteq \mathcal{D}$, and $\mathrm{Irr}(A; \Gamma, \sim; \mathcal{D})$ the set of isomorphism classes of irreducible modules of $\HC(A; \Gamma, \sim; \mathcal{D})$, we obtain as a corollary our first main theorem:

\begin{theorem}
Suppose $\Gamma$ is a Harish-Chandra block subalgebra of $A$.
\begin{enumerate}[{\rm (i) }]
    \item $\HC(A; \Gamma, \sim) = \bigoplus_{\mathcal{D} \in \big(\cfsclasses\big)/\Delta} \HC(A; \Gamma, \sim; \mathcal{D})$
    \item $\mathrm{Irr}(A; \Gamma, \sim) = \bigsqcup_{\mathcal{D} \in \big(\cfsclasses\big)/\nabla} \mathrm{Irr}(A; \Gamma, \sim; \mathcal{D})$
\end{enumerate}
\end{theorem}

\subsection{Sufficient Conditions and Relation to Drozd-Futorny-Ovsienko}\label{subsec:dfosubalgebras}
In this section, we compare Harish-Chandra block modules to the Harish-Chandra modules of \cite{drozd1994harish}.

\begin{definition}[{\cite{drozd1994harish}}]
We say $\Gamma$ is \emph{quasicentral} (in $A$) if for all $a \in A$, $\Gamma a \Gamma$ is finitely generated as a left and right $\Gamma$-module.
\end{definition}

The following result shows that any Harish-Chandra subalgebra in the sense of \cite{drozd1994harish} is a Harish-Chandra block subalgebra with respect to the equality relation.

\begin{proposition}\label{prop:dfosuff}
Let $\Gamma$ be a subalgebra of $A$. If $\Gamma$ satisfies the following:
\begin{enumerate}[{\rm (i) }]
    \item $\Gamma$ is noetherian
    \item $\Gamma$ is quasicommutative
    \item $\Gamma$ is quasicentral in $A$
\end{enumerate}
then $\Gamma$ is a Harish-Chandra block subalgebra of $A$ with respect to the equality relation. Furthermore, for any $\frak{m} \in \cfs(\Gamma)$ and $m \geq 1$: \begin{equation}\label{eq:dforefinement}\Supp(A/A\frak{m}^m) \subseteq X(\frak{m}) := \bigcup_{a \in A} \{\frak{n}: S_\frak{n} \in \Fact(\frac{\Gamma a \Gamma}{\Gamma a \frak{m}})\}.\end{equation}
\end{proposition}

\begin{remark}
The set $X(\frak{m})$ of \ref{eq:dforefinement} is introduced in \cite{drozd1994harish}. The inclusion in \ref{eq:dforefinement} indicates that Proposition \ref{prop:blockspan} is a refinement of \cite[Prop.~14]{drozd1994harish}.
\end{remark}

\begin{proof}
By \cite[Prop.~14]{drozd1994harish}, if $v \in V(\frak{m})$ for an $A$-module $V$, then $Av \subseteq \bigoplus_{\frak{n} \in X(\frak{m})} V(\frak{n})$.

Take $V = A/A\frak{m}^m$ for any $\frak{m} \in \cfs(\Gamma)$ and $m \geq 1$. Then $1 + A \frak{m}^m \in V(\frak{m})$, so $A/A\frak{m}^m = A.(1 + A \frak{m}^m) \subseteq \bigoplus_{\frak{n} \in X(\frak{m})} V(\frak{n})$. So $A/A\frak{m}^m$ is a Harish-Chandra block module with respect to $=$ for any $\frak{m} \in \cfs(\Gamma)$ and $m \geq 1$, hence $\Gamma$ is a Harish-Chandra block subalgebra of $A$ with respect to $=$.
\end{proof}

The definition of a Harish-Chandra subalgebra given in \cite{drozd1994harish} is often easier to check. We provide an analogous set of criteria for checking that $\Gamma$ is a Harish-Chandra block subalgebra of $A$ with respect to $\sim$.

\begin{definition}
    We say $\Gamma$ is quasinoetherian (with respect to $\sim$) if for every $B \in \cfsclasses$ and $\frak{m} \in \fwords{B}$, $\Gamma / \frak{m}$ is finite-dimensional.
\end{definition}

\begin{proposition}
    Let $\Gamma$ be a subalgebra of $A$. If $\Gamma$ satisfies the following:
    \begin{enumerate}[{\rm (i) }]
        \item $\Gamma$ is quasinoetherian with respect to $\sim$
        \item $\Gamma$ is quasicentral in $A$
    \end{enumerate}
    then $\Gamma$ is a Harish-Chandra block subalgebra of $A$ with respect to the $\mathrm{Ext}$ relation.
\end{proposition}
\begin{proof}
    Let $B \in \cfsclasses$, $\frak{m} \in \fwords{B}$, and $V = A/A\frak{m}$. Let $a \in A$. Consider the $\Gamma$-module map $\frac{\Gamma a \Gamma}{\Gamma a \frak{m}} \rightarrow A/A\frak{m}$. Since $\Gamma$ is quasinoetherian and quasicentral, $\frac{\Gamma a \Gamma}{\Gamma a \frak{m}}$ is finite-dimensional, and therefore a block module by Proposition \ref{prop:extmodule}. Now $a + A\frak{m}$ is in the image of this map, and therefore in $\bigoplus_{B \in \cfsclasses} V(B)$ by Proposition \ref{prop:blockmaps}.
\end{proof}

\subsection{Strong Harish-Chandra Block Modules and Block Subalgebras}\label{sec:strong}

In this section, we define strong block modules, which generalizes the requirement of finite-dimensional weight spaces.

\begin{definition}
A $\Gamma$-module $V$ is called a \emph{strong block module} (with respect to $\sim$) if $V$ is a block module with respect to $\sim$ and for each $B \in \cfsclasses$, there is $\frak{m} \in \fwords{B}$ so that $\frak{m}V(B) = 0$. If $V$ is also an $A$-module, we call $A$ a \emph{strong Harish-Chandra block module} (with respect to $\sim$). Let $\SHC(A; \Gamma, \sim)$ be the full subcategory of strong Harish-Chandra block modules with respect to $\sim$.
\end{definition}

The next lemma shows that a form of Fitting's Lemma holds for strong block modules.

\begin{lemma}[Fitting's Lemma]\label{Fitting}
If $V$ is a strong block module, then for each $B \in \cfsclasses$ we have the decomposition $$V = V(B) \oplus V'$$ where $\frak{m}V' = V'$ for all $\frak{m} \in \fwords{B}$ and $\frak{m}V(B) = 0$ for some $\frak{m} \in \fwords{B}$.
\end{lemma}
\begin{proof}
Let $V = V(B) \oplus V'$ where $V' = \bigoplus_{C \neq B} V(C)$. Then $\frak{m} V(B) = 0$ for some $\frak{m} \in \fwords{B}$ by the definition of a strong block module. We show that $\frak{m}V' = V'$ for any $\frak{m} \in \fwords{B}$. It is sufficient to show this for any $\frak{m} \in B$.

Let $\frak{m} \in B$ and $v \in V'$. By Proposition \ref{SES}, $\Gamma v$ is a block module with $\Supp(\Gamma v) \subseteq \Supp(V')$, and consequently $\Gamma v/\frak{m} v$ is also a block module with $\Supp(\Gamma v / \frak{m}v) \subseteq \Supp(\Gamma v)$. In particular, $[\frak{m}]_\sim = B \not \in \Supp(\Gamma v / \frak{m} v)$. So by Lemma \ref{blockfactors}$, S_\frak{m} \not \in \Fact(\Gamma v / \frak{m} v)$ hence $\Gamma v = \frak{m} v \subseteq \frak{m}V'$. Since $v$ was arbitrary, we have $\frak{m}V' = V'$.
\end{proof}

We define a special case of a Harish-Chandra block subalgebra that we will need in order to define the category $\mathcal{A}$ in Section \ref{sec:category}.

\begin{definition}
We say $\Gamma$ is a \emph{strong Harish-Chandra block subalgebra} of $A$ (with respect to $\sim$) if:
\begin{itemize}
    \item $A/A\frak{m}$ is a strong Harish-Chandra block module for every $B \in \cfsclasses$ and $\frak{m} \in \fwords{B}$
    \item $A/\frak{l}A$ is a strong Harish-Chandra block module (as a right module) for every $C \in \cfsclasses$ and $\frak{l} \in \fwords{C}$
\end{itemize}
\end{definition}

\begin{remark}
    When $C \in \cfsclasses$ and $\frak{l} \in \fwords{C}$, we will always interpret $A/\frak{l}A$ as a right module. If $D \in \cfsclasses$, then $(\frac{A}{\frak{l}A})(D)$ should be interpreted as the block space of $A/\frak{l}A$ as a right module.
\end{remark}

\subsection{Upper Triangular Subalgebras}\label{subsec:hctri}

Consider again the example of the upper triangular algebra. Let $\{e, f, h\}$ denote the standard basis for $A_0 = U(\frak{sl}_2(\mathbb{C}))$ and $c \in Z(A_0)$ denote the Casimir element.

It will be useful to first see that $\Gamma_0 = \mathbb{C}[h,c]$ is a strong Harish-Chandra block subalgebra of $A_0$ with respect to the equality relation. Realizing $A_0$ as a generalized Weyl algebra \cite{bavula1992generalized} we have $A_0 = \bigoplus_{n \in \mathbb{Z}} X^n\Gamma_0$ where $$X^n = \begin{cases}
e^n & n > 0\\
1 & n = 0\\
f^{-n} & n < 0
\end{cases}.$$ Taking $\sigma$ to be the automorphism of $\Gamma_0$ that fixes $\mathbb{C}[c]$ and maps $h$ to $h-2$, we have $\frak{n}X^n = X^n\sigma^{-n}(\frak{n})$ for any $\frak{n} \in \cfs(\Gamma_0)$. For any $\frak{m} = (h - \lambda, c - \mu) \in \cfs(\Gamma_0)$ and $m \geq 0$, let $M = A_0/A_0\frak{m}^m$. For $n \in \mathbb{Z}$, $$\frac{X^n\Gamma_0 + A_0 \frak{m}^m}{A_0\frak{m}^m}$$ is the block space $M(\sigma^{n}(\frak{m}))$ and has finite basis $$\{X^n(h-\lambda)^i(c-\mu)^j + A_0 \frak{m}^m\}_{\substack{i,j \geq 0\\ i+j < m}}.$$ This demonstrates that $A_0/A_0\frak{m}^m$ is a strong Harish-Chandra block module with respect to the equality relation. A similar argument will show that $A_0/\frak{m}^m A_0$ is also a strong Harish-Chandra block module.

Now we will take $\mathbb{C}x$ to be the one dimensional $(\Gamma_0,\Gamma_0)$-bimodule where $h,c,e,f$ act by $0$, so that $$\Gamma = \begin{bmatrix} \Gamma_0 & \mathbb{C}x \\ 0 & \Gamma_0 \end{bmatrix} \leq \begin{bmatrix} A_0 & \mathbb{C}x \\ 0 & A_0 \end{bmatrix} = A$$ are both algebras under addition and multiplication of matrices. Let $\sim$ be the $\mathrm{Ext}$ relation on $\cfs(\Gamma)$. Recall that \begin{equation}\cfsclasses = \{ \{\overline{\frak{z}}, \underline{\frak{z}}\} \} \cup \{ \{\overline{\frak{m}}\}\}_{\frak{m} \in \cfs(\Gamma_0)\setminus \{\frak{z}\}} \cup \{\{\underline{\frak{m}}\} \}_{\frak{m} \in \cfs(\Gamma_0)\setminus \{\frak{z}\}}.\end{equation}
We now aim to show that $\Gamma$ is a strong Harish-Chandra block subalgebra of $A$ with respect to $\sim$.

Let $\frak{m} \in \cfs(\Gamma_0)\setminus \{\frak{z}\}$ and $m \geq 0$. The module $$A/A\overline{\frak{m}}^m = \begin{bmatrix} A_0 & \mathbb{C}x \\ 0 & A_0 \end{bmatrix} / \begin{bmatrix} A_0\frak{m}^m & \mathbb{C}x \\ 0 & A_0 \end{bmatrix} \cong \begin{bmatrix} A_0/A_0\frak{m}^m & 0\\ 0 & 0 \end{bmatrix},$$ so $A/A\overline{\frak{m}}^m$ is a strong Harish-Chandra block module with respect to the equality relation, and therefore a strong Harish-Chandra block module with respect to $\sim$, with $\Blocks(A/A\overline{\frak{m}}^m) = \{[\overline{\frak{n}}]_\sim: \frak{n} \in \Supp(A_0/A_0\frak{m}^m)\}$. Similarly, $$A/A\underline{\frak{m}}^m \cong \begin{bmatrix} 0 & 0\\ 0 & A_0/A_0\frak{m}^m \end{bmatrix}$$ is a strong Harish-Chandra block module, with $\Blocks(A/A\underline{\frak{m}}^m) = \{[\underline{\frak{n}}]_\sim: \frak{n} \in \Supp(A_0/A_0\frak{m}^m)\}$. An identical computation shows that $A/\overline{\frak{m}}^m A$ and $A/\underline{\frak{m}}^m A$ are also strong Harish-Chandra block modules.

To complete the proof that $\Gamma$ is a strong Harish-Chandra block subalgebra of $A$, we need to check that $A/A \frak{p}$ and $A/\frak{p} A$ are strong Harish-Chandra block modules for all $\frak{p} \in \fwords{\{\overline{\frak{z}}, \underline{\frak{z}}\}}$. Note that if $n,m \geq 1$, then $$\overline{\frak{z}}^n \underline{\frak{z}}^m = \left[\begin{array}{cc} \frak{z}^n & 0 \\
    0 & \frak{z}^m
\end{array} \right].$$ So $(\overline{\frak{z}}^n \underline{\frak{z}}^m)\overline{\frak{z}} = \overline{\frak{z}}^{n+1} \underline{\frak{z}}^m$ and $\underline{\frak{z}}(\overline{\frak{z}}^n \underline{\frak{z}}^m) = \overline{\frak{z}}^n \underline{\frak{z}}^{m+1}$. So $\fwords{\{\overline{\frak{z}}, \underline{\frak{z}}\}} = \{\overline{\frak{z}}^n \underline{\frak{z}}^m: n,m \geq 1\} \cup \{\underline{\frak{z}}^m\overline{\frak{z}}^n: n,m \geq 0\}$.

When $n,m \geq 1$, $$A/A\underline{\frak{z}}^m\overline{\frak{z}}^n \cong \left[\begin{array}{cc}
    A_0/A_0 \frak{z}^n & 0 \\
    0 & A_0/A_0 \frak{z}^m
\end{array}\right] = \left[\begin{array}{cc}
    A_0/A_0 \frak{z}^n & 0 \\
    0 & 0
\end{array}\right] \oplus \left[\begin{array}{cc}
    0 & 0 \\
    0 & A_0/A_0 \frak{z}^m
\end{array}\right]$$ is a strong Harish-Chandra block module, with 
$$\Blocks(A/A\underline{\frak{z}}^m\overline{\frak{z}}^n) = \{[\overline{\frak{n}}]_\sim: \frak{n} \in \Supp(A_0/A_0\frak{z}^n)\} \cup \{[\underline{\frak{m}}]_\sim: \frak{m} \in \Supp(A_0/A_0\frak{z}^m)\}.$$

Now let $n,m \geq 0$. Say $N = A_0/A_0 \frak{z}^n$ and $M = A_0/A_0 \frak{z}^m$ with $i,j \geq 0$ so that $\frak{z}^iN(\frak{z}) = 0$ and $\frak{z}^jM(\frak{z}) = 0$. Recall that $\frak{z}M(\frak{m}) = M(\frak{m})$ for $\frak{m} \in \cfs(\Gamma_0) \setminus \{\frak{z}\}$. So \begin{align*}V = A/A\overline{\frak{z}}^n \underline{\frak{z}}^m &\cong \left[\begin{array}{cc}
    A_0/A_0 \frak{z}^n & \mathbb{C}x \\
    0 & A_0/A_0 \frak{z}^m
\end{array}\right] \\
&= \Bigg( \bigoplus_{\frak{n} \in \cfs(\Gamma_0) \setminus \{\frak{z}\}} \left[\begin{array}{cc}
    N(\frak{n}) & 0 \\
    0 & 0
\end{array}\right] \Bigg) \oplus \left[\begin{array}{cc}
    N(\frak{z}) & \mathbb{C}x \\
    0 & M(\frak{z})
\end{array}\right] \oplus \Bigg( \bigoplus_{\frak{m} \in \cfs(\Gamma_0) \setminus \{\frak{z}\}} \left[\begin{array}{cc}
    0 & 0 \\
    0 & M(\frak{m})
\end{array}\right] \Bigg)\\
&= \Bigg( \bigoplus_{\frak{n} \in \cfs(\Gamma_0) \setminus \{\frak{z}\}} V(\overline{\frak{n}}) \Bigg) \oplus V(\{\overline{\frak{z}}, \underline{\frak{z}}\}) \oplus \Bigg( \bigoplus_{\frak{m} \in \cfs(\Gamma_0) \setminus \{\frak{z}\}} V(\underline{\frak{m}}) \Bigg)\end{align*} where $\overline{\frak{z}}^{i+1}\underline{\frak{z}}^{j+1}V(\{\overline{\frak{z}}, \underline{\frak{z}}\}) = 0$.
So $A/A\overline{\frak{z}}^n \underline{\frak{z}}^m$ is a strong Harish-Chandra block module with $$\Blocks(A/A\overline{\frak{z}}^n\underline{\frak{z}}^m) = \{[\overline{\frak{n}}]_\sim: \frak{n} \in \Supp(A_0/A_0\frak{z}^n)\} \cup \{\{\overline{\frak{z}}, \underline{\frak{z}}\}\} \cup \{[\underline{\frak{m}}]_\sim: \frak{m} \in \Supp(A_0/A_0\frak{z}^m)\}.$$ As before, similar arguments will suffice for $A/\frak{p}A$ when $\frak{p} \in \fwords{\{\overline{\frak{z}}, \underline{\frak{z}}\}}$.

\section{Category Equivalence}\label{sec:category}
In this section, we establish categories which are equivalent to $\HC(A; \Gamma, \sim)$ or $\SHC(A; \Gamma, \sim)$. Throughout this section, we will need to assume that $\Gamma$ is a strong Harish-Chandra block subalgebra of $A$.
\subsection{$\mathcal{A}$}

We will define a category $\mathcal{A}$ with objects $\cfsclasses$ and morphisms \begin{equation}\mathcal{A}(B, C) = \lim_{\frak{n} \in \fwords{C} \frak{m} \in \fwords{B}} \frac{A}{\frak{n}A + A\frak{m}} = \varprojlim \frac{A}{\frak{n}A + A\frak{m}},\end{equation} with the goal of associating Harish-Chandra block modules with topologically-enriched functors from $\mathcal{A}$ to $\mathrm{Vect}_{\Bbbk}$. The following lemma will be important for defining composition in $\mathcal{A}$.

\begin{lemma}\label{doublecosetiso}
Suppose $\Gamma$ is a strong Harish-Chandra block subalgebra of $A$. Let $B, C, D \in \cfsclasses,$ $\frak{m} \in \fwords{B}$ and $\frak{l} \in \fwords{D}$. There is $\frak{n} \in \fwords{C}$ so that \begin{equation}\frac{A}{\frak{l} A + A \frak{n}} \cong (\frac{A}{\frak{l} A})(C) \qquad \text{ and } \qquad \frac{A}{\frak{n} A + A \frak{m}} \cong (\frac{A}{A \frak{m}})(C)\end{equation} as $(\Gamma, \Gamma)$-bimodules.
\end{lemma}

\begin{proof}
Firstly, for any $\frak{n} \in \fwords{C}$, we have natural maps:
\[
\begin{tikzcd}
  A/A\frak{m} \arrow[r] \arrow[d]& \frac{A}{\frak{n} A + A \frak{m}} \arrow[d, equals]\\
  (\frac{A}{A \frak{m}})(C) \arrow[r] & \big(\frac{A}{\frak{n} A + A \frak{m}}\big)(C)
\end{tikzcd}
\]
Let $\frak{n} \in \fwords{C}$ so that $\frak{n}(\frac{A}{A \frak{m}})(C) = 0$. Since $A/A \frak{m}$ is a Harish-Chandra block module, there are $A_P \supseteq A \frak{m}$ so that $A_P/A\frak{m} = (\frac{A}{A \frak{m}})(P)$ for $P \in \cfsclasses$. We have $\frak{n} A_C \subseteq A\frak{m}$ and $\frak{n}A_P + A\frak{m} = A_P$ for $P \neq C$, by Lemma \ref{Fitting}. Since $A = \sum A_P$, we have $\frak{n} A + A \frak{m} = \sum_{P \neq C} A_P$ which maps to zero under $A/A\frak{m} \rightarrow (\frac{A}{A \frak{m}})(C)$. This induces a map $\frac{A}{\frak{n} A + A \frak{m}} \rightarrow (\frac{A}{A \frak{m}})(C)$.
So we have:
\[
\begin{tikzcd}
  & \arrow[dl, twoheadrightarrow] A/A\frak{m} \arrow[d, twoheadrightarrow] \arrow[dr, twoheadrightarrow]\\
  (\frac{A}{A \frak{m}})(C) \arrow[r] & \frac{A}{\frak{n} A + A\frak{m}} \arrow[r] & (\frac{A}{A \frak{m}})(C)
\end{tikzcd}
\]
By uniqueness, these maps are inverses. It will be useful in Appendix \ref{sec:appendix} that the isomorphism $\frac{A}{\frak{n} A + A \frak{m}} \cong (\frac{A}{A \frak{m}})(C)$ is natural in $\frak{m}$ and $\frak{n}$.
A symmetric argument shows that $\frac{A}{\frak{l} A + A \frak{n}} \cong (\frac{A}{\frak{l} A})(C)$.
\end{proof}

We can now define the composition in $\mathcal{A}$. First, for $\alpha \in \mathcal{A}(B,C)$, let $\alpha_{\frak{m},\frak{n}}$ denote the image under the projection $\mathcal{A}(B,C) \rightarrow \frac{A}{\frak{n} A + A \frak{m}}$.

\begin{definition}\label{def:comp}
The composition map $\mathcal{A}(C, D) \otimes_\Gamma \mathcal{A}(B, C) \rightarrow \mathcal{A}(B, D)$ is given as follows:
\begin{enumerate}
\item For each $\frak{m} \in \fwords{B}$ and $\frak{l} \in \fwords{D}$, pick $\frak{n} \in \fwords{C}$ so that $\frac{A}{\frak{l} A + A \frak{n}} \cong (\frac{A}{\frak{l} A})(C)$ and $\frac{A}{\frak{n} A + A \frak{m}} \cong (\frac{A}{A \frak{m}})(C)$.
\item Given $\alpha \in \mathcal{A}(B, C)$ and $\beta \in \mathcal{A}(C, D)$, take $a_0, b_0 \in A$ so that:
\begin{align*}
    \alpha_{\frak{m},\frak{n}} &= a_0 + \frak{n}A + A \frak{m} & & a_0 + A \frak{m} \in (\frac{A}{A \frak{m}})(C)\\
    \beta_{\frak{n},\frak{l}} &= b_0 + \frak{l}A + A \frak{m} & & b_0 + \frak{l} A \in (\frac{A}{\frak{l} A})(C)
\end{align*}
\item Then define $(\beta \circ \alpha)_{\frak{m},\frak{l}} = b_0a_0 + \frak{l} A + A \frak{m}$.
\end{enumerate}
\end{definition}

\begin{proposition}\label{prop:comp}
    The proposed composition map is a well defined $(\Gamma, \Gamma)$-bimodule map, and makes $\mathcal{A}$ into a category.
\end{proposition}

The proof of this proposition is provided in Appendix \ref{sec:appendix}.

\subsection{$\amod$}

We now enrich the categories $\mathcal{A}$ and $\mathbf{Vect}_\Bbbk$ over the category of topological vector spaces, so that we may consider the category of enriched functors from $\mathcal{A}$ to $\mathbf{Vect}_\Bbbk$.

Since $\mathcal{A}(B, C) = \varprojlim \frac{A}{\frak{n} A + A \frak{m}}$, we give $\mathcal{A}(B, C)$ the limit topology. (That is, the topology with the fewest open sets such that the projections $\mathcal{A}(B,C) \rightarrow \frac{A}{\frak{n} A + A \frak{m}}$ are continuous for all $\frak{m} \in \fwords{B}$ and $\frak{n} \in \fwords{C}$ when the $\frac{A}{\frak{n} A + A \frak{m}}$ are given the discrete topology.) We will enrich $\mathbf{Vect}_\Bbbk$ over the category of topological vector spaces with respect to two topologies. Firstly, we will take the discrete topology on $\Hom_\Bbbk(V, W)$. Secondly, note that any vector space $V$ is the colimit of its finite-dimensional subspaces. This gives a realization of $\Hom_\Bbbk(V,W)$ as the limit of $\Hom_\Bbbk(U,W)$ across finite-dimensional subspaces $U$ of $V$. We will make use of the limit topology this gives $\Hom_\Bbbk(V,W)$.

The following result is standard for limits of topological vector spaces.

\begin{lemma}\label{discretelemma}
Let $X$ be a topological vector space under the discrete topology. Then a linear map $f: \mathcal{A}(B, C) \rightarrow X$ is continuous if and only if there are $\frak{m} \in \fwords{B}$ and $\frak{n} \in \fwords{C}$ so that $f$ factors:
\[
\begin{tikzcd}
  \mathcal{A}(B, C) \arrow[d] \arrow[r]& X\\
  \frac{A}{\frak{n}A + A \frak{m}} \arrow[ur, dashed]
\end{tikzcd}\]
\end{lemma}

With the two enrichments of $\mathbf{Vect}_\Bbbk$, we now define profinite and discrete $\mathcal{A}$-modules as enriched functors from $\mathcal{A}$ to $\mathbf{Vect}_\Bbbk$. We will refer to a functor from $\mathcal{A}$ to $\mathbf{Vect}_\Bbbk$ as an $\mathcal{A}$-module.

\begin{definition}\phantom{X}
\begin{enumerate}
    \item A \emph{profinite $\mathcal{A}$-module} is a functor $F: \mathcal{A} \rightarrow \mathbf{Vect}_\Bbbk$ such that each $F_{B,C}: \mathcal{A}(B, C) \rightarrow \Hom_\Bbbk(F(B), F(C))$ is a continuous linear map with respect to the limit topology on $\Hom_\Bbbk(F(B), F(C))$. Let $\amod$ denote the full subcategory of profinite $\mathcal{A}$-modules.
    \item A \emph{discrete $\mathcal{A}$-module} is a functor $F: \mathcal{A} \rightarrow \mathbf{Vect}_\Bbbk$ such that each $F_{B,C}: \mathcal{A}(B, C) \rightarrow \Hom_\Bbbk(F(B), F(C))$ is a continuous linear map with respect to the discrete topology on $\Hom_\Bbbk(F(B), F(C))$. Let $\samod$ denote the full subcategory of discrete $\mathcal{A}$-modules.
\end{enumerate}
 For ease of notation, when $a \in A$, we write $F_{B, C}(a)$ to mean $F$ applied to the image $(a)$ of $a$ in $\mathcal{A}(B, C)$.
\end{definition}

\subsection{Equivalence}

In this section, we will prove the equivalences of Theorem \ref{thm:main2}. First, we show that to every profinite $\mathcal{A}$-module $F$, there is an associated Harish-Chandra block module $\mathcal{H}(F)$. Furthermore, if $F$ is a discrete $\mathcal{A}$-module, then $\mathcal{H}(F)$ is a strong Harish-Chandra block module.

\begin{proposition}\label{functorobjects}
Let $\Gamma$ be a strong Harish-Chandra block subalgebra of $A$. If $F \in \amod$, then $\mathcal{H}(F) := \bigoplus F(B)$ is a Harish-Chandra block module with action $a.v := \sum_C F_{B, C}(a)v$ for $a \in A$ and $v \in F(B)$. Furthermore, if $F \in \samod$, then $\mathcal{H}(F)$ is a strong Harish-Chandra block module.
\end{proposition}

\begin{proof}
In order for this action to be well defined, we must prove that $F_{B,C}(a)v = 0$ for all but finitely many $C \in \cfsclasses$.

The image of $\Gamma$ in $\mathcal{A}(B, C)$ is zero for $C \neq B$, so we can define $\gamma.v = F_{B, B}(\gamma)v$ for each $\gamma \in \Gamma$ and $v \in F(B)$. We have $F_{B, B}(\gamma_1 \gamma_2) = F_{B, B}(\gamma_1.(\gamma_2)) = F_{B, B}((\gamma_1) \circ (\gamma_2)) = F_{B, B}(\gamma_1) \circ F_{B, B}(\gamma_2)$, making each $F(B)$ a $\Gamma$-module.

Let $v \in F(B)$. By Lemma \ref{discretelemma}, there are $\frak{m}_1, \frak{m}_2 \in \fwords{B}$ so that the following diagram factors:
\[
\begin{tikzcd}
  A \arrow[r] \arrow[dr] & \mathcal{A}(B, B) \arrow[r] \arrow[dr] \arrow[d]& \Hom_\Bbbk(F(B), F(B)) \arrow[d, "- \circ \iota"]\\
  & \frac{A}{\frak{m}_1A + A \frak{m}_2} \arrow[r, dashed] & \Hom_\Bbbk(\Bbbk v, F(B))
\end{tikzcd}
\]
So $\frak{m}_2$ gives the zero action on $\Bbbk v$, making $\mathcal{H}(F)$ a block module.

If $F \in \samod$, then for $B \in \cfsclasses$ there are $\frak{m}_1, \frak{m}_2 \in \fwords{B}$ so that:
\[
\begin{tikzcd}
  A \arrow[r] \arrow[dr] & \mathcal{A}(B, B) \arrow[r] \arrow[d] & \Hom_\Bbbk(F(B), F(B))\\
  & \frac{A}{\frak{m}_1A + A \frak{m}_2} \arrow[ur, dashed]
\end{tikzcd}
\]
So $\frak{m}_2F(B) = 0$, making $\mathcal{H}(F)$ a strong block module.

Now $\Hom_\Bbbk(F(B), F(C))$ is a $(\Gamma, \Gamma)$-bimodule, and $F_{B, C}(\gamma_1.\alpha.\gamma_2) = F_{B, C}((\gamma_1) \circ \alpha \circ (\gamma_2)) = F_{C, C}(\gamma_1) \circ F_{B, C}(\alpha) \circ F_{B, B}(\gamma_2)  = \gamma_1.F_{B, C}(\alpha).\gamma_2$. So $F_{B, C}$ is a map of $(\Gamma, \Gamma)$-bimodules.

Let $v \in F(B)$. Let $\frak{m}' \in \fwords{B}$ so that $\frak{m}' v = 0$. Note that $\Hom_\Bbbk(\Gamma v, F(C))$ is a $(\Gamma, \Gamma)$-bimodule with $\Hom_\Bbbk(\Gamma v, F(C)).\frak{m}' = 0$. Again, let $\frak{m} \in \fwords{B}$ and $\frak{n} \in \fwords{C}$ so that:
\[
\begin{tikzcd}
  A \arrow[r] \arrow[dr] & \mathcal{A}(B, C) \arrow[r] \arrow[dr] \arrow[d]& \Hom_\Bbbk(F(B), F(C)) \arrow[d, "- \circ \iota"]\\
  & \frac{A}{\frak{n}A + A \frak{m}} \arrow[r, dashed] & \Hom_\Bbbk(\Bbbk v, F(C))
\end{tikzcd}
\]
So for any $a \in \frak{n}A$, $F_{B,C}(a)(v) = 0$. In particular, for any $a \in \frak{n}A$ and $\gamma \in \Gamma$, we have $0 = F_{B,C}(a\gamma)(v) = F_{B,C}((a).\gamma)(v) = (F_{B,C}(a).\gamma)(v) = F_{B,C}(a)(\gamma v)$. Hence the following diagram of $(\Gamma, \Gamma)$-bimodules factors:
\begin{equation}\label{eq:diagram}
\begin{tikzcd}
  A \arrow[r] \arrow[dr] & \mathcal{A}(B, C) \arrow[r] \arrow[dr] \arrow[d]& \Hom_\Bbbk(F(B), F(C)) \arrow[d, "- \circ \iota"]\\
  & \frac{A}{\frak{n}A + A \frak{m}'} \arrow[r, dashed] & \Hom_\Bbbk(\Gamma v, F(C))
\end{tikzcd}
\end{equation}
since $F_{B, C}(\frak{n}A) \circ \iota = 0$, but also $F_{B, C}(A\frak{m}') \circ \iota = \Big( F_{B, C}(A) \circ \iota \Big).\frak{m}' = 0$. Of particular importance is that now $\frak{m}'$ depends only on $B$ (and not on $C$).
Note that for any $C \in \cfsclasses$ and $\frak{n} \in \fwords{C}$, we have
\[
\begin{tikzcd}
  A \arrow[d] \arrow[r] & \frac{A}{\frak{n}A + A\frak{m}'}\\
  \frac{A}{A\frak{m}'} \arrow[ur] \arrow[r] & \big(\frac{A}{A\frak{m}'}\big)(C) \arrow[u]
\end{tikzcd}
\]
So given $a \in A$, we have $a \in \frak{n}A + A\frak{m}'$ if $\pi_C(a + A\frak{m}') = 0$. Now in \ref{eq:diagram}, $F_{B, C}(a) \circ \iota = 0$ for all but finitely many $C \in \cfsclasses$, and the proposed action is well defined.

Now let $v \in F(B)$ and $a,b \in A$. Suppose $\frak{m} \in \fwords{B}$ so that $\frak{m} v = 0$. We have just shown that $a.v = \sum_{i \in I} F_{B, C_i}(a)(v)$ where $I = \{i: \pi_{C_i}(a + A \frak{m}) \neq 0 \}$. We can write $a + A\frak{m} = \sum_{i \in I} a_i + A\frak{m}$ for some $a_i + A\frak{m} \in (\frac{A}{A\frak{m}})(C_i)$.
Given $D \in \cfsclasses$, take $\frak{l} \in \fwords{D}$ so that:

\[
\begin{tikzcd}
  \mathcal{A}(B, D) \arrow[r] \arrow[d] & \Hom_\Bbbk(F(B), F(D)) \arrow[d] \\
  \frac{A}{\frak{l} A + A \frak{m}} \arrow[r] & \Hom_\Bbbk(\Bbbk v, F(D))
\end{tikzcd}
\]
Write $b + \frak{l} A = \sum_{j \in J} b_j + \frak{l} A$ for some $b_j + \frak{l} A \in (\frac{A}{\frak{l} A})(C_j)$.

For each $k \in I \cup J$, let $\frak{n}_k \in \fwords{C_k}$ so that $\frac{A}{\frak{l} A + A \frak{n}_k} \cong (\frac{A}{\frak{l} A})(C_k)$ and $\frac{A}{\frak{n}_k A + A \frak{m}} \cong (\frac{A}{A \frak{m}})(C_k)$. Then for distinct $i,j$, we have $b_j \in \frak{l} A + A \frak{n}_i$ and $a_i \in \frak{n}_j A + A \frak{m}$, thus $b_j a_i \in \frak{l} A + A \frak{m}$. So

\begin{align*}
    ba + \frak{l} A + A \frak{m} &= (\sum_{j \in J} b_j)(\sum_{i \in I} a_i) + \frak{l} A + A \frak{m}\\
    &= \sum_{i \in I, j\in J} b_j a_i + \frak{l} A + A \frak{m}\\
    &= \sum_{i \in I} b_ia_i + \frak{l} A + A \frak{m}
\end{align*}
Furthermore, $b - b_k \in \frak{l} A + A \frak{n}_k$ and $a - a_k \in \frak{n}_k A + A \frak{m}$. So $(b) \circ (a) \mapsto b_ka_k + \frak{l} A + A \frak{m}$ under the composition map $\mathcal{A}(C_k, D) \otimes_\Gamma \mathcal{A}(B, C_k) \rightarrow \frac{A}{\frak{l} A + A \frak{m}}$.
Thus $\pi_D(ba.v) = F_{B, D}(ba)(v) = F_{B, D}( \sum_{i \in I} (b) \circ (a) )(v) = \sum_{i \in I} F_{C_i, D}(b) \circ F_{B, C_i}(a) (v) = \sum_{i \in I} \pi_D(b.F_{B, C_i}(a)(v))= \pi_D(b.\sum_{i \in I} F_{B, C_i}(a)(v)) = \pi_D(b.(a.v))$.
So we have an action of $A$ on $\mathcal{H}(F)$, as desired.
\end{proof}

The association $F \mapsto \mathcal{H}(F)$ is a functor which gives the two category equivalences in the following theorem.

\begin{theorem}\label{thm:equivalence}
If $\Gamma$ is a strong Harish-Chandra block subalgebra of $A$, then:
\begin{enumerate}[{\rm (i) }]
    \item $\amod \cong \HC(A; \Gamma, \sim)$
    \item $\samod \cong \SHC(A; \Gamma, \sim)$
\end{enumerate}
\end{theorem}

\begin{proof}
We establish a functor $\mathcal{H}$ from $\amod$ to $\HC(A; \Gamma, \sim)$ that will give the desired equivalences.
If $F \in \amod$, then $\mathcal{H}(F) = \bigoplus_{B \in \cfsclasses} F(B)$ is a Harish-Chandra block module by Proposition \ref{functorobjects}.

Given a natural transformation $\theta: F \rightarrow G$, we have:
\[
\begin{tikzcd}
F(B) \arrow[r, "\theta_B"] \arrow[d, "F_{B, C}(\alpha)" left]& G(B) \arrow[d, "G_{B, C}(\alpha)"]\\
F(C) \arrow[r, "\theta_C"] & G(C)
\end{tikzcd}
\]
for each $\alpha \in \mathcal{A}(B, C)$. In particular, this diagram commutes for $\alpha = (a)$ when $a \in A$. Since $\pi_C(a.v) = \pi_C(\sum_D F_{B, D}(a)v) = F_{B, C}(a)v$ for $v \in F(B)$, this diagram becomes:
\[
\begin{tikzcd}
F(B) \arrow[r, "\theta_B"] \arrow[d, "\pi_C(a.-)" left]& G(B) \arrow[d, "\pi_C(a.-)"]\\
F(C) \arrow[r, "\theta_C"] & G(C)
\end{tikzcd}
\]
so that $\mathcal{H}(\theta) := \bigoplus \theta_B$ is a map of Harish-Chandra block modules, by Proposition \ref{naturaltransformation}.

Furthermore, $\mathcal{H}(id) = \bigoplus id = id$ and $\mathcal{H}(\theta \circ \sigma) = \bigoplus (\theta_B \circ \sigma_B) = \bigoplus \theta_B \circ \bigoplus \sigma_B = \mathcal{H}(\theta) \circ \mathcal{H}(\sigma)$, so that $\mathcal{H}$ is a functor.
Clearly, $\theta \mapsto \bigoplus \theta_B$ is injective.

Again, by Proposition \ref{naturaltransformation}, every map between the Harish-Chandra block modules $\mathcal{H}(F) = \bigoplus F(B)$ and $\mathcal{H}(G) = \bigoplus G(B)$ is a direct sum of linear maps $\theta_B$ which for every $a \in A$ satisfies:
\[
\begin{tikzcd}
F(B) \arrow[r, "\theta_B"] \arrow[d, "F_{B, C}(a) = \pi_C(a.-)" left]& G(B) \arrow[d, "\pi_C(a.-) = G_{B, C}(a)"]\\
F(C) \arrow[r, "\theta_C"] & G(C)
\end{tikzcd}
\]
If $v \in F(B)$, then there exist $\frak{m} \in \fwords{B}$ and $\frak{n} \in \fwords{C}$ inducing the following commutative diagram:
\[
\begin{tikzcd}
\Hom_\Bbbk(\Bbbk v, F(C)) & \arrow[l] \mathcal{A}(B, C) \arrow[r] \arrow[d] & \Hom_\Bbbk(\theta_B(\Bbbk v), G(C))\\
& \frac{A}{\frak{n}A + A \frak{m}} \arrow[ul, dashed] \arrow[ur, dashed]
\end{tikzcd}
\]
Now if $\alpha \in \mathcal{A}(B, C)$ and $\alpha_{\frak{m},\frak{n}} = a + \frak{n} A + A \frak{m}$, then $F_{B, C}(\alpha)(v) = F_{B, C}(a)(v) = \pi_C(a.v)$ and $G_{B, C}(\alpha)(\theta_B(v)) = G_{B, C}(a)(\theta_B(v)) = \pi_C(a.\theta_B(v)) = \theta_C(\pi_C(a.v))$. So we have the commutative diagram:
\[
\begin{tikzcd}
F(B) \arrow[r, "\theta_B"] \arrow[d, "F_{B, C}(\alpha)" left]& G(B) \arrow[d, "G_{B, C}(\alpha)"]\\
F(C) \arrow[r, "\theta_C"] & G(C)
\end{tikzcd}
\]
for each $\alpha \in \mathcal{A}(\frak{m}, \frak{n})$, hence $\theta = \{\theta_B\}$ is a natural transformation. So the assignment $\theta \mapsto \bigoplus \theta_B$ is surjective.

So $\mathcal{H}$ is full and faithful. We now show it is essentially surjective, finishing the equivalence.
To each $V \in \HC(A; \Gamma, \sim)$, we construct a functor $\mathcal{V} \in \amod$. Define $\mathcal{V}(B) = V(B)$.
Let $B, C \in \cfsclasses$ and $W \leq V(B)$ finite-dimensional. Let $\frak{m} \in \fwords{B}$ and $\frak{n} \in \fwords{C}$ so that $\frak{m} W = 0$ and $\frak{n} \pi_C( A.W ) = 0$ ($\frak{n}$ exists since $\Gamma$ is a strong Harish-Chandra block subalgebra of $A$ and $W$ is finite-dimensional). So we have a linear map $\frac{A}{\frak{n} A + A \frak{m}} \rightarrow \Hom_\Bbbk(W, V(C))$ given by $a + \frak{n} A + A \frak{m} \mapsto \pi_C(a.-)$ since $\pi_C((\frak{n}A + A \frak{m}).-) = \frak{n}\pi_C(A.-) + \pi_C(A\frak{m}.-) = 0$. So we consider the composition
\[
\begin{tikzcd}
\mathcal{A}(B, C) \arrow[r] \arrow[d] & \Hom_\Bbbk(W, V(C))\\
\frac{A}{\frak{n} A + A \frak{m}} \arrow[ur]
\end{tikzcd}
\]
which we will show induces a continuous linear map $\mathcal{A}(B, C) \rightarrow \Hom_\Bbbk(V(B), V(C))$.

If $\frak{m}' \in \fwords{B}$ and $\frak{n}' \in \fwords{C}$ so that $\frak{m}' \subseteq \frak{m}$ and $\frak{n}' \subseteq \frak{n}$, then:

\[
\begin{tikzcd}
& \frac{A}{\frak{n}' A + A \frak{m}'} \arrow[dd] \arrow[dr]\\
  \mathcal{A}(B, C) \arrow[ur] \arrow[dr] & & \Hom_\Bbbk(W, V(C))\\
  & \frac{A}{\frak{n} A + A \frak{m}} \arrow[ur]
\end{tikzcd}
\]
so that $\mathcal{A}(B, C) \rightarrow \Hom_\Bbbk(W, V(C))$ is well defined.
If $W' \supseteq W$ is a finite-dimensional subspace of $V(\frak{m})$, and $\frak{m}' \in \fwords{B}$ and $\frak{n}' \in \fwords{C}$ with $\frak{m}' \subseteq \frak{m}$ and $\frak{n}' \subseteq \frak{n}$ so that $\frak{m}'W' = 0$ and $\frak{n}'\pi_\frak{n}(A.W') = 0$, then:

\[
\begin{tikzcd}
& \frac{A}{\frak{n}' A + A \frak{m}'} \arrow[dd] \arrow[r] & \Hom_\Bbbk(W', V(C)) \arrow[dd]\\
  \mathcal{A}(B, C) \arrow[ur] \arrow[dr]\\
  & \frac{A}{\frak{n} A + A \frak{m}} \arrow[r] & \Hom_\Bbbk(W, V(C))
\end{tikzcd}
\]
This induces a continuous linear map $\mathcal{V}_{B, C}$ into the limit $\Hom_\Bbbk(V(B), V(C))$:

\[
\begin{tikzcd}
\mathcal{A}(B, C) \arrow[r, "\mathcal{V}_{B,C}"] \arrow[d] \arrow[dr] & \Hom_\Bbbk(V(B), V(C)) \arrow[d] \\
\frac{A}{\frak{n} A + A \frak{m}} \arrow[r] & \Hom_\Bbbk(W, V(C))
\end{tikzcd}
\]
Furthermore, if $V$ is a strong Harish-Chandra block module, then we may choose the same $\frak{m} \in \fwords{B}$ and $\frak{n} \in \fwords{C}$ for every $W \leq V(B)$ finite-dimensional, giving the continuity diagram:

\[
\begin{tikzcd}
\mathcal{A}(B, C) \arrow[r] \arrow[d] & \Hom_\Bbbk(V(B), V(C))\\
\frac{A}{\frak{n} A + A \frak{m}} \arrow[ur]
\end{tikzcd}
\]

So if we can show the maps $\mathcal{V}_{B, C}$ make $\mathcal{V}$ a functor, then we will have $\mathcal{V} \in \amod$ (or in $\samod$, if $V$ is a strong Harish-Chandra block module.)
Firstly, $\mathcal{V}_{B, B}(1) = \pi_B(1.-) = id_{V(B)}$.

Let $\beta \in \mathcal{A}(C, D)$, $\alpha \in \mathcal{A}(B, C)$, and $v \in V(B)$. Take $\frak{m} \in \fwords{B}$, $\frak{n} \in \fwords{C}$, and $\frak{l} \in \fwords{D}$ so that we have the continuity diagrams:

\[
\begin{tikzcd}
\mathcal{A}(C, D) \arrow[r] \arrow[d] & \Hom_\Bbbk(\mathcal{V}_{B, C}(\alpha)(\Bbbk v), V(D)) & \mathcal{A}(B, C) \arrow[r] \arrow[d] & \Hom_\Bbbk(\Bbbk v, V(C))\\
\frac{A}{\frak{l}A + A \frak{n}} \arrow[ur] & & \frac{A}{\frak{n}A + A \frak{m}} \arrow[ur]\\
& \mathcal{A}(B, D) \arrow[r] \arrow[d] & \Hom_\Bbbk(\Bbbk v, V(D))\\
& \frac{A}{\frak{l} A + A \frak{m}} \arrow[ur]
\end{tikzcd}
\]
In fact, we have these for any $\frak{n}' \in \fwords{C}$ with $\frak{n}' \subseteq \frak{n}$, so assume $\frak{n}$ is also large enough that we have:
\begin{align*}
    \alpha_{\frak{m},\frak{n}} = a_0 + \frak{n} A + A \frak{m} && a_0 + A\frak{m} \in (\frac{A}{A \frak{m}})(B)\\
    \beta_{\frak{n},\frak{l}} = b_0 + \frak{l} A + A \frak{n} && b_0 + \frak{l} A \in (\frac{A}{\frak{l} A})(B) \\
    (\beta \circ \alpha)_{\frak{m},\frak{l}} = b_0a_0 + \frak{l} A + A \frak{m}
\end{align*}
Then $\frak{n} a_0 v \subseteq A \frak{m} v = 0$, so $\mathcal{V}_{B,D}(\beta \circ \alpha)(v) = \pi_D(b_0a_0.v) = \pi_D(b_0.\pi_
C(a_0.v)) = \mathcal{V}_{C,D}(\beta)( \mathcal{V}_{B,C}(\alpha)(v))$.

So $\mathcal{V} \in \amod$. Now $\mathcal{H}(\mathcal{V}) = \bigoplus \mathcal{V}(B) = \bigoplus V(B)$ and for $v \in V(B)$, $a.v = \sum \pi_C(a.v) = \sum \mathcal{V}_{B,C}(a)v$, so $\mathcal{H}(\mathcal{V}) \cong V$.
Hence $\mathcal{H}$ is fully faithful and essentially surjective, establishing the desired equivalence.
\end{proof}

\subsection{Remarks on Categories in Drozd-Futorny-Ovsienko}
The definition of the categories $\mathcal{A}$, $\amod$, and $\samod$ were inspired by the definitions in \cite{drozd1994harish}. In fact, the definitions of $\mathcal{A}$ and $\samod$ coincide with the definitions in \cite{drozd1994harish} when $\sim$ is the equality relation. There are two major differences to note.

Firstly, we require that $\Gamma$ is a \emph{strong} Harish-Chandra block subalgebra of $A$. It is unclear whether the composition defined in Definition \ref{def:comp} and in \cite{drozd1994harish} is well-defined under weaker assumptions.

Secondly, \cite[Thm.~17]{drozd1994harish} claims that the category of Harish-Chandra modules, denoted in this paper by $\HC(A; \Gamma, =)$, is equivalent to the category of discrete $\mathcal{A}$-modules. However, we have amended this in Theorem \ref{thm:equivalence} as the category of discrete $\mathcal{A}$-modules is equivalent to the category $\SHC(A; \Gamma, =)$ of \emph{strong} Harish-Chandra modules. The category of Harish-Chandra modules is instead equivalent to the category $\amod$ of \emph{profinite} $\mathcal{A}$-modules which we have introduced for this purpose.

Many of the generalizations of classical Harish-Chandra modules discussed in the introduction are assumed to have finite-dimensional weight spaces. Since any Harish-Chandra module in the sense of \cite{drozd1994harish} having finite-dimensional generalized weight spaces is a strong Harish-Chandra block module with respect to the equality relation, the ambiguity in \cite{drozd1994harish} is created by the possibility of infinite-dimensional generalized weight spaces.

%Consider the example of $\Gamma = \mathbb{C}[\partial]$ in the Weyl algebra $A_1(\mathbb{C}) = \mathbb{C}\langle x, \partial: \partial x - x \partial = 1 \rangle$. In this case, $\Gamma$ is noetherian, commutative, and quasicentral in $A$, and therefore a Harish-Chandra subalgebra of $A$ in both the sense of \cite{drozd1994harish} and this paper (see Proposition \ref{prop:dfosuff}). However, $\mathcal{A}(\frak{m}, \frak{m}) = 0$ for any $\frak{m} \in \cfs(\Gamma)$.

%The failure here may be attributed to the fact that $\Gamma$ is not a strong Harish-Chandra block subalgebra of $A$.

\section{Irreducible Harish-Chandra Block Modules}\label{sec:irr}
Given $B \in \cfsclasses$, let $\hat{\Gamma}_B = \lim_{\frak{m} \in \fwords{B}} \Gamma/\frak{m}$. For $\frak{m} \in \fwords{B}$, let $\pi_\frak{m}: \hat{\Gamma}_B \rightarrow \Gamma/\frak{m}$ denote the canonical projection. We will leverage the equivalence of the previous section in order to prove the following theorem.

\begin{theorem}\label{thm:irreducibles}
    Let $\Gamma$ be a strong Harish-Chandra block subalgebra of $A$ with respect to $\sim$, and let $B \in \cfsclasses$. Suppose $B$ is finite, $\Gamma$ is noetherian, and $\mathcal{A}(B,B)$ is finitely generated as a left and right $\hat{\Gamma}_B$-module. Then:
    \begin{enumerate}[{\rm (i) }]
        \item There are finitely many isoclasses of simple Harish-Chandra block modules $V$ with $B \in \Supp(V)$.
        \item If $V$ is a simple Harish-Chandra block module $V$, then $V(B)$ is finite-dimensional.
    \end{enumerate}
\end{theorem}
We will begin by proving some lemmas about the Jacobson radical of $\hat{\Gamma}_B$.
\begin{lemma}\label{lem:completerad}
    Given $B \in \cfsclasses$, the Jacobson radical $\mathrm{Rad}(\hat{\Gamma}_B)$ is equal to $\bigcap_{\frak{m} \in B} \ker(\pi_\frak{m})$.
\end{lemma}
\begin{proof}
    We begin by showing that $\bigcap_{\frak{m} \in B}\ker(\pi_\frak{m}) \subseteq \mathrm{Rad}(\hat{\Gamma}_B)$. Let $x \in \bigcap_{\frak{m} \in B}\ker(\pi_\frak{m})$. Say $1-x = (y_{\frak{m}} + \frak{m})_{\frak{m} \in \fwords{B}}$. Note that for any $0 < i \leq k$ and $\frak{m}_1, \ldots, \frak{m}_k \in B$ we may chase $x$ across the commutative diagram:
    \[\begin{tikzcd}
        & \hat{\Gamma}_B \arrow[dl] \arrow[dr] & \\
        \Gamma/\frak{m}_1 \cdots \frak{m}_k \arrow[rr] & & \Gamma/\frak{m}_i
    \end{tikzcd}\]
    to observe that $1 - y_{\frak{m}_1 \cdots \frak{m}_k} \in \frak{m}_i$ and therefore \begin{equation}\label{eq:quasiproduct}
        (1 - y_{\frak{m}_1 \cdots \frak{m}_k})^k \in \frak{m}_1 \cdots \frak{m}_k.
    \end{equation}
    For $k \geq 0$, let $p_k(t)$ be the polynomial: \begin{equation}
        p_k(t) = \frac{1 - (1-t)^k}{t}.
    \end{equation} We claim $$(p_k(y_{\frak{m}_1 \cdots \frak{m}_k}) + \frak{m}_1 \cdots \frak{m}_k)_{\substack{k \geq 0\\ \frak{m}_1, \ldots, \frak{m}_k \in B}}$$ is in $\hat{\Gamma}_B$ and is the inverse of $1-x$. Consequently every member of the ideal $\bigcap_{\frak{m} \in B}\ker(\pi_\frak{m})$ is quasiregular, and therefore in $\mathrm{Rad}(\hat{\Gamma}_B)$.
    
    Let $k \geq 1$ and $\frak{m}_1,\ldots,\frak{m}_k \in B$. Take $i \in [k]$ and $\frak{n} = \frak{m}_1 \cdots \frak{m}_{i-1} \frak{m}_{i+1} \cdots \frak{m}_k$. Then
    \begin{align}
        p_k(y_{\frak{m}_1 \cdots \frak{m}_k}) - p_{k-1}(y_\frak{n}) &= p_k(y_{\frak{m}_1 \cdots \frak{m}_k}) - p_k(y_\frak{n}) + p_k(y_\frak{n}) - p_{k-1}(y_\frak{n}) \notag \\
        &= p_k(y_{\frak{m}_1 \cdots \frak{m}_k}) - p_k(y_\frak{n}) + (1-y_\frak{n})^{k-1} \label{eq:polyidentity}
    \end{align}
    where the equality of \ref{eq:polyidentity} follows from the polynomial identity: $$p_k(t) - p_{k-1}(t) = (1-t)^{k-1}.$$ By \ref{eq:quasiproduct}, we have $(1-y_\frak{n})^{k-1} \in \frak{n}$. By chasing $p_k(1-x)$ across the commutative diagram:
    \[\begin{tikzcd}
        & \hat{\Gamma}_B \arrow[dl] \arrow[dr] & \\
        \Gamma/\frak{m}_1 \cdots \frak{m}_k \arrow[rr] & & \Gamma/\frak{n}
    \end{tikzcd}\] we have $p_k(y_{\frak{m}_1 \cdots \frak{m}_k}) - p_k(y_\frak{n}) \in \frak{n}$.
    Thus we have $$(p_k(y_{\frak{m}_1 \cdots \frak{m}_k}) + \frak{m}_1 \cdots \frak{m}_k)_{\substack{k \geq 0\\ \frak{m}_1, \ldots, \frak{m}_k \in B}}\in\hat{\Gamma}_B.$$ Finally, for any $k \geq 0$, we have the polynomial identity: $$1-tp_k(t) = 1 - p_k(t)t = (1-t)^{k}$$
    so by \ref{eq:quasiproduct} we have:
    \begin{align*}
        1-y_{\frak{m}_1 \cdots \frak{m}_k}p_k(y_{\frak{m}_1 \cdots \frak{m}_k}) = 1-p_k(y_{\frak{m}_1 \cdots \frak{m}_k})y_{\frak{m}_1 \cdots \frak{m}_k} &= (1 - y_{\frak{m}_1 \cdots \frak{m}_k})^k\\
        &\in \frak{m}_1 \cdots \frak{m}_k
    \end{align*}
    hence $1-x$ has inverse $$(p_k(y_{\frak{m}_1 \cdots \frak{m}_k}) + \frak{m}_1 \cdots \frak{m}_k)_{\substack{k \geq 0\\ \frak{m}_1, \ldots, \frak{m}_k \in B}}$$ as desired.

    Conversely, for any $\frak{m} \in B$ we have that $\pi_\frak{m}(\mathrm{Rad}(\hat{\Gamma}_B)) \subseteq \mathrm{Rad}(\Gamma/\frak{m})$ since $\pi_\frak{m}$ is surjective. Now $\mathrm{Rad}(\Gamma/\frak{m}) = 0$, since $\Gamma/\frak{m}$ is a finite-dimensional simple algebra. So for any $\frak{m} \in B$, we have $\mathrm{Rad}(\hat{\Gamma}_B) \subseteq \ker(\pi_\frak{m})$, and therefore $\mathrm{Rad}(\hat{\Gamma}_B) \subseteq \bigcap_{\frak{m} \in B}\ker(\pi_\frak{m})$, completing the proof.
    \end{proof}

    \begin{corollary}\label{cor:finiterad}
        If $B \in \cfsclasses$ is finite, then $\hat{\Gamma}_B/\mathrm{Rad}(\hat{\Gamma}_B)$ is finite-dimensional.
    \end{corollary}
    \begin{proof}
    Note that $\{\ker(\pi_\frak{m})\}_{\frak{m} \in B}$ are distinct maximal two-sided ideals of $\hat{\Gamma}_B$, hence:
    \begin{align*}
        \hat{\Gamma}_B/\mathrm{Rad}(\hat{\Gamma}_B) &= \hat{\Gamma}_B/\bigcap_{\frak{m} \in B}\ker(\pi_\frak{m})\\
        &\cong \prod_{\frak{m} \in B} \hat{\Gamma}_B/\ker(\pi_\frak{m})\\
        &\cong \prod_{\frak{m} \in B} \Gamma/\frak{m}.
    \end{align*}
    Consequently, $\hat{\Gamma}_B/\mathrm{Rad}(\hat{\Gamma}_B)$ is finite-dimensional.
\end{proof}

\begin{lemma}\label{lem:radfinitepower}
    Let $B \in \cfsclasses$. If $B$ is finite, and $\Gamma$ is noetherian, then $\hat{\Gamma}_B/(\mathrm{Rad}(\hat{\Gamma}_B))^n$ is finite dimensional for every $n$.
\end{lemma}
\begin{proof}
    Say $B = \{\frak{m}_1, \ldots, \frak{m}_k\}$ and $\frak{b} = \frak{m}_1 \cap \cdots \cap \frak{m}_k$. We will denote the $\frak{b}$-adic completion of $\Gamma$ by $\Tilde{\Gamma} = \lim_n \Gamma/\frak{b}^n$ and the $\frak{b}$-adic completion of a $\Gamma$-module $M$ by $\Tilde{M} = \lim_n M/\frak{b}^nM$. We will show $\hat{\Gamma}_B \cong \Tilde{\Gamma}$, and use standard properties of the $\frak{b}$-adic completion to prove the proposition.

    We have induced maps defined as follows:
    \[\begin{tikzcd}
        \hat{\Gamma}_B \arrow[r, dashed, "\varphi"] \arrow[d] & \Tilde{\Gamma}\arrow[d] & \Tilde{\Gamma} \arrow[r, dashed, "\psi"] \arrow[d] &\hat{\Gamma}_B \arrow[d] \\
        \Gamma/(\frak{m}_1\cdots\frak{m}_k)^n \arrow[r] & \Gamma/\frak{b}^n & \Gamma/\frak{b}^n \arrow[r] & \Gamma/\frak{m}_{i_1}\cdots \frak{m}_{i_n}
    \end{tikzcd}\]
    The composition $\psi \circ \varphi$ is the identity by the diagram:
    \[\begin{tikzcd}
        \hat{\Gamma}_B \arrow[r, "\varphi"] \arrow[d] & \Tilde{\Gamma}\arrow[d] \arrow[r, "\psi"] \arrow[d] &\hat{\Gamma}_B \arrow[d] \\
        \Gamma/(\frak{m}_1\cdots\frak{m}_k)^n \arrow[r] & \Gamma/\frak{b}^n \arrow[r] & \Gamma/\frak{m}_{i_1}\cdots \frak{m}_{i_n}
    \end{tikzcd}\]
    and similarly, the composition $\varphi \circ \psi$ is the identity by the diagram:
    \[\begin{tikzcd}
        \Tilde{\Gamma} \arrow[r, dashed, "\psi"] \arrow[d] & \hat{\Gamma}_B \arrow[d] \arrow[r, dashed, "\varphi"] \arrow[d] & \Tilde{\Gamma}\arrow[d]\\
        \Gamma/\frak{b}^{nk} \arrow[r] & \Gamma/(\frak{m}_1\cdots\frak{m}_k)^n \arrow[r] & \Gamma/\frak{b}^n
    \end{tikzcd}\]
    so that $\hat{\Gamma}_B \cong \Tilde{\Gamma}$. Furthermore, we have that $\Tilde{\frak{b}} = \ker(\Tilde{\Gamma} \rightarrow \Gamma/\frak{b})$ is contained in $\mathrm{Rad}(\Tilde{\Gamma})$, since
    \[\begin{tikzcd}
        \Tilde{\Gamma} \arrow[r, dashed, "\psi"] \arrow[d] &\hat{\Gamma}_B \arrow[d, "\pi_{\frak{m}_i}"] \\
        \Gamma/\frak{b} \arrow[r] & \Gamma/\frak{m}_i
    \end{tikzcd}\]
    for every $i$ \big(and $\mathrm{Rad}(\hat{\Gamma}_B) = \bigcap_{\frak{m} \in B}\ker(\pi_\frak{m})$ by Lemma \ref{lem:completerad}\big). Conversely, we have the surjective map $\Tilde{\Gamma} \rightarrow \Gamma/\frak{b}$ so that $\mathrm{Rad}(\Tilde{\Gamma})$ maps into $\mathrm{Rad}(\Gamma/\frak{b})$. Since $\Gamma/\frak{b} \cong \Gamma/\frak{m}_1 \times \cdots \times \Gamma/\frak{m}_k$ is a finite-dimensional semisimple algebra, $\mathrm{Rad}(\Gamma/\frak{b}) = 0$ and $\mathrm{Rad}(\Tilde{\Gamma}) \subseteq \Tilde{\frak{b}}$.

    To complete the proof, we show $\Tilde{\Gamma}/(\mathrm{Rad}(\Tilde{\Gamma}))^n = \Tilde{\Gamma}/\Tilde{\frak{b}}^n$ is finite dimensional for every $n$. Since $\Gamma$ is noetherian, $\Tilde{\frak{b}}^n = \Tilde{\frak{b}^n}$ (see, for example, \cite{atiyah1969introduction}). Now $\Tilde{\frak{b}^n} = \ker(\Tilde{\Gamma} \rightarrow \Gamma/\frak{b}^n)$, so $\Tilde{\Gamma}/\Tilde{\frak{b}^n} \cong \Gamma/\frak{b}^n$ which is finite-dimensional since $\Gamma$ is noetherian and $\Gamma/\frak{b}$ is finite-dimensional.
    \end{proof}

 We make use of the following well-known result. A full proof can be found in \cite[Thm.~18]{drozd1994harish}, for the case where $\sim$ is the equality relation.

\begin{proposition}\label{prop:modulecorrespondence}
    Let $B \in \cfsclasses$. There is a one to one correspondence between the isomorphism classes of simple $\mathcal{A}(B,B)$-modules and the isomorphism classes of simple $\mathcal{A}$-modules $F$ with $F(B) \neq 0$.
\end{proposition}
\begin{proof}
    If $F$ is a simple $\mathcal{A}$-module, then $F(B)$ is a simple $\mathcal{A}(B,B)$-module. Given a simple $\mathcal{A}(B,B)$-module $V$, take $G = \mathcal{A}(B, -) \otimes_{\mathcal{A}(B,B)} V$. Then $G$ is an $\mathcal{A}$-module with a unique maximal submodule $K$. The quotient $H = G/K$ is a simple $\mathcal{A}$-module with $H(B) \cong V$.
\end{proof}

The last thing we will need to prove Theorem \ref{thm:irreducibles} is the following lemma, which implies that simple profinite $\mathcal{A}$-modules are simple as $\mathcal{A}$-modules.

\begin{lemma}\label{lem:subtop}
    If $F$ is a profinite (discrete) $\mathcal{A}$-module, and $G$ is an $\mathcal{A}$-submodule of $F$, then $G$ is a profinite (discrete) $\mathcal{A}$-module.
\end{lemma}

\begin{proof}
    Let $F$ be a profinite $\mathcal{A}$-module. Suppose $G$ is an $\mathcal{A}$-submodule of $F$, and $B,C \in \cfsclasses$. Then the natural inclusion map $G \hookrightarrow F$ yields the commutative diagram:
    \[\begin{tikzcd}
        \mathcal{A}(B,C) \arrow[r] \arrow[d] & \Hom_\Bbbk(F(B),F(C)) \arrow[d]\\
        \Hom_\Bbbk(G(B),G(C)) \arrow[r] & \Hom_\Bbbk(G(B), F(C))
    \end{tikzcd}\]
    If $W \leq G(B)$ is finite-dimensional, then we may choose $\frak{m} \in \fwords{B}$ and $\frak{n} \in \fwords{C}$ so that
    \begin{equation}\begin{tikzcd}
        \mathcal{A}(B,C) \arrow[r] \arrow[d] & \Hom_\Bbbk(F(B),F(C)) \arrow[d]\\
        \frac{A}{\frak{n}A + A\frak{m}} \arrow[r] & \Hom_\Bbbk(W, F(C))
    \end{tikzcd}\end{equation}
    commutes, since $F$ is profinite. Finally, note that
    \[\begin{tikzcd}
        \Hom_\Bbbk(F(B), F(C)) \arrow[rr] \arrow[dr] & & \Hom_\Bbbk(G(B), F(C)) \arrow[dl]\\
        & \Hom_\Bbbk(W, F(C))
    \end{tikzcd}\]
    commutes. Combining these, we receive the commutative diagram:
    \[\begin{tikzcd}
        A\arrow[rrr] \arrow[d] & & & \frac{A}{\frak{n}A + A\frak{m}}\arrow[dddd, "f"]\\
        \mathcal{A}(B,C) \arrow[urrr, "\pi"]\arrow[rr]\arrow[dd, "G_{B,C}"] & & \Hom_\Bbbk(F(B), F(C))\arrow[dd]\arrow[dddr]\\
        \\
        \Hom_\Bbbk(G(B), G(C))\arrow[d, "i"]\arrow[rr] & & \Hom_\Bbbk(G(B), F(C))\arrow[dr]\\
        \Hom_\Bbbk(W, G(C))\arrow[rrr, "j"] & & & \Hom_\Bbbk(W, F(C))
    \end{tikzcd}\]
    In particular, the outer square factors:
    \[\begin{tikzcd}
        A \arrow[r, twoheadrightarrow] \arrow[d] & \frac{A}{\frak{n}A + A\frak{m}} \arrow[d, "f"]\arrow[dl, dashed, "g"]\\
        \Hom_\Bbbk(W, G(C))\arrow[r, hookrightarrow, "j"] & \Hom_\Bbbk(W, F(C))
    \end{tikzcd}\]
    We now have that $j \circ i \circ G_{B,C} = f \circ \pi = j \circ g \circ \pi$. Since $j$ is injective, we have the commutative diagram:
    \[\begin{tikzcd}
        \mathcal{A}(B,C) \arrow[r, "G_{B,C}"] \arrow[d, "\pi"] & \Hom_\Bbbk(G(B),G(C)) \arrow[d, "i"]\\
        \frac{A}{\frak{n}A + A\frak{m}} \arrow[r, "g"] & \Hom_\Bbbk(W,G(C))
    \end{tikzcd}\]
    and thus $G$ is a profinite $\mathcal{A}$-module.

    If $F$ is a discrete $\mathcal{A}$-module, then the above argument can be simplified to show that $G$ is also a discrete $\mathcal{A}$-module (simply replace $W$ with $G(B)$).
\end{proof}

\begin{proof}[Proof of Theorem \ref{thm:irreducibles}]
    Suppose $B$ is finite, $\Gamma$ is noetherian, and $\mathcal{A}(B,B)$ is finitely generated as a left and right $\hat{\Gamma}_B$-module. By Corollary \ref{cor:finiterad}, $\hat{\Gamma}_B/\mathrm{Rad}(\hat{\Gamma}_B)$ is finite-dimensional. Hence $\frac{\mathcal{A}(B,B)}{\mathcal{A}(B,B)\mathrm{Rad}(\hat{\Gamma}_B)}$ is also finite-dimensional. So there is $n$ such that $(\mathrm{Rad}(\hat{\Gamma}_B))^n \mathcal{A}(B,B) \subseteq \mathcal{A}(B,B)\mathrm{Rad}(\hat{\Gamma}_B)$. If $I$ is a maximal right ideal of $\mathcal{A}(B,B)$, then $(\mathrm{Rad}(\hat{\Gamma}_B))^n \mathcal{A}(B,B) \subseteq I$, otherwise $\mathcal{A}(B,B) = I + (\mathrm{Rad}(\hat{\Gamma}_B))^n \mathcal{A}(B,B) \subseteq I + \mathcal{A}(B,B)\mathrm{Rad}(\hat{\Gamma}_B)$ (which would imply $I = \mathcal{A}(B,B)$, by Nakayama's Lemma). Consequently, $(\mathrm{Rad}(\hat{\Gamma}_B))^n \mathcal{A}(B,B) \subseteq \mathrm{Rad}(\mathcal{A}(B,B))$. By Lemma \ref{lem:radfinitepower}, $\hat{\Gamma}_B/(\mathrm{Rad}(\hat{\Gamma}_B))^n$ is finite-dimensional. So $\frac{\mathcal{A}(B,B)}{(\mathrm{Rad}(\hat{\Gamma}_B))^n\mathcal{A}(B,B)}$ and $\frac{\mathcal{A}(B,B)}{\mathrm{Rad}(\mathcal{A}(B,B))}$ are also finite-dimensional.

    Since $\frac{\mathcal{A}(B,B)}{\mathrm{Rad}(\mathcal{A}(B,B))}$ is finite-dimensional, there are only finitely many isomorphism classes of simple $\mathcal{A}(B,B)$-modules, and they are all finite-dimensional. The correspondence of Proposition \ref{prop:modulecorrespondence} then gives that there are only finitely many isomorphism classes of simple $\mathcal{A}$-modules $F$ with $F(B) \neq 0$, and each has $F(B)$ finite-dimensional. We may say the same thing about profinite or discrete $\mathcal{A}$-modules, as Lemma \ref{lem:subtop} implies that any simple profinite $\mathcal{A}$-module is a simple $\mathcal{A}$-module. The theorem then follows from the equivalence of Theorem \ref{thm:equivalence}.
\end{proof}

\subsection{Upper Triangular Subalgebra}
We return one final time to the upper triangular algebras of Section \ref{subsec:hctri}. We note that $\Gamma$ is noetherian, and each block of $\sim$ (the $\mathrm{Ext}$ relation) is finite. We will show that for each $B \in \cfsclasses$, $\mathcal{A}(B,B)$ is finitely generated as a left and right $\hat{\Gamma}_B$-module, and we may apply Theorem \ref{thm:irreducibles}.

We first note that for any $\frak{m} \in \cfs(\Gamma_0)$ and $i,j \geq 0$, we have the map: \begin{equation}\label{eq:usliso}\Gamma_0 \rightarrow \frac{A_0}{\frak{m}^i A_0 + A_0 \frak{m}^j}\end{equation} with kernel $(\frak{m}^i A_0 + A_0 \frak{m}^j) \cap \Gamma_0$. Again realizing $A_0$ as a generalized Weyl algebra \cite{bavula1992generalized}, we have $A_0$ equal to both $\bigoplus_{n \in \mathbb{Z}} \Gamma_0 X^n$ and $\bigoplus_{n \in \mathbb{Z}} X^n\Gamma_0$ where $$X^n = \begin{cases}
e^n & n > 0\\
1 & n = 0\\
f^{-n} & n < 0
\end{cases}.$$ Taking $\sigma$ to be the automorphism of $\Gamma_0$ that fixes $\mathbb{C}[c]$ and maps $h$ to $h-2$, we can compute
\begin{align*}
    (\frak{m}^i A_0 + A_0 \frak{m}^j) \cap \Gamma_0 &= \Big(\bigoplus_{n \in \mathbb{Z}} \frak{m}^i X^n + \bigoplus_{n \in \mathbb{Z}} X^n\frak{m}^j\Big) \cap \Gamma_0\\
    &= \Big(\bigoplus_{n \in \mathbb{Z}} \big(\frak{m}^i + \sigma^n(\frak{m}^j)\big)X^n\Big) \cap \Gamma_0 \\
   &= \big(\frak{m}^i + \sigma^0(\frak{m}^j)\big)X^0\\
   &= \frak{m}^i + \frak{m}^j.
\end{align*}
We may show that \ref{eq:usliso} is surjective by demonstrating that $X^n \in \frak{m}^i A_0 + A_0 \frak{m}^j$ for $n \neq 0$. Say $\frak{m} = (h - \lambda, c - \mu)$. Consider the action of $\Gamma_0 \otimes_\mathbb{C} \Gamma_0^\text{op}$ on $\frac{A_0}{\frak{m}^i A_0 + A_0 \frak{m}^j}$. For $n \neq 0$, the element $(h-\lambda - 2n)\otimes 1 + 1 \otimes (h-\lambda)$ is invertible modulo $\frak{m} \otimes_\mathbb{C} \Gamma_0 + \Gamma_0 \otimes_\mathbb{C} \frak{m}$ and therefore is invertible modulo $(\frak{m} \otimes_\mathbb{C} \Gamma_0 + \Gamma_0 \otimes_\mathbb{C} \frak{m})^{i+j} \subseteq \frak{m}^i \otimes_\mathbb{C} \Gamma_0 + \Gamma_0 \otimes_\mathbb{C} \frak{m}^j$. So $(h-\lambda - 2n)\otimes 1 + 1 \otimes (h-\lambda)$ acts invertibly on $X^n + \frak{m}^i A_0 + A_0 \frak{m}^j$, but also annihilates $X^n + \frak{m}^i A_0 + A_0 \frak{m}^j$. Thus $X^n \in \frak{m}^i A_0 + A_0 \frak{m}^j$.

Now for any $\frak{m} \in \cfs(\Gamma_0)$ and $i,j \geq 0$,
$$\frac{A_0}{\frak{m}^i A_0 + A_0 \frak{m}^j} \cong \frac{\Gamma_0}{\frak{m}^i + \frak{m}^j}.$$ So when $\frak{m} \neq \frak{z}$,
$$
\frac{A}{\overline{\frak{m}}^i A + A \overline{\frak{m}}^j} \cong \left[\begin{array}{cc}
    \frac{A_0}{\frak{m}^i A_0 + A_0 \frak{m}^j} & 0 \\
    0 & 0
\end{array}\right] \cong \left[\begin{array}{cc}
    \frac{\Gamma_0}{\frak{m}^i +  \frak{m}^j} & 0 \\
    0 & 0
\end{array}\right] \cong \frac{\Gamma}{\overline{\frak{m}}^i + \overline{\frak{m}}^j}
$$
and hence $$\mathcal{A}(\{\overline{\frak{m}}\}, \{\overline{\frak{m}}\}) \cong \lim_{i,j} \frac{\Gamma}{\overline{\frak{m}}^i + \overline{\frak{m}}^j}$$ which is cyclic as both a left and right $\hat{\Gamma}_{\overline{\frak{m}}}$-module. Similarly, $\mathcal{A}(\{\underline{\frak{m}}\}, \{\underline{\frak{m}}\})$ is cyclic as both a left and right $\hat{\Gamma}_{\underline{\frak{m}}}$-module.

For the block $B = \{\overline{\frak{z}}, \underline{\frak{z}}\}$, recall that every $\frak{p} \in \fwords{\{\overline{\frak{z}}, \underline{\frak{z}}\}}$ is either equal to $\overline{\frak{z}}^{n(\frak{p})} \underline{\frak{z}}^{m(\frak{p})}$ where $n(\frak{p}),m(\frak{p}) \geq 1$ or equal to $\underline{\frak{z}}^{m(\frak{p})}\overline{\frak{z}}^{n(\frak{p})}$ where $n(\frak{p}),m(\frak{p}) \geq 0$. Furthermore, $$\overline{\frak{z}}^n \underline{\frak{z}}^m = \left[\begin{array}{cc} \frak{z}^n & 0 \\
    0 & \frak{z}^m
\end{array} \right] \qquad (n,m \geq 1)$$ and 
$$\underline{\frak{z}}^m \overline{\frak{z}}^n = \left[\begin{array}{cc} \frak{z}^n & \mathbb{C}x \\
    0 & \frak{z}^m
\end{array} \right] \qquad (n,m \geq 0).$$
Let $\frak{p}, \frak{l} \in \fwords{\{\overline{\frak{z}}, \underline{\frak{z}}\}}$. If $\frak{p},\frak{l} \in \{\overline{\frak{z}}^n \underline{\frak{z}}^m: n,m \geq 1\}$, then
$$\frac{A}{\frak{p}A + A\frak{l}} \cong \left[\begin{array}{cc}
    \frac{A_0}{\frak{z}^{n(\frak{p})} A_0 + A_0 \frak{z}^{n(\frak{l})}} & \mathbb{C}x \\
    0 & \frac{A_0}{\frak{z}^{m(\frak{p})} A_0 + A_0 \frak{z}^{m(\frak{l})}}
\end{array}\right] \cong \left[\begin{array}{cc}
    \frac{\Gamma_0}{\frak{z}^{n(\frak{p})} + \frak{z}^{n(\frak{l})}} & \mathbb{C}x \\
    0 & \frac{\Gamma_0}{\frak{z}^{m(\frak{p})} + \frak{z}^{m(\frak{l})}}
\end{array}\right] \cong \frac{\Gamma}{\frak{p} + \frak{l}}.$$
Otherwise,
$$\frac{A}{\frak{p}A + A\frak{l}} \cong \left[\begin{array}{cc}
    \frac{A_0}{\frak{z}^{n(\frak{p})} A_0 + A_0 \frak{z}^{n(\frak{l})}} & 0 \\
    0 & \frac{A_0}{\frak{z}^{m(\frak{p})} A_0 + A_0 \frak{z}^{m(\frak{l})}}
\end{array}\right] \cong \left[\begin{array}{cc}
    \frac{\Gamma_0}{\frak{z}^{n(\frak{p})} + \frak{z}^{n(\frak{l})}} & 0 \\
    0 & \frac{\Gamma_0}{\frak{z}^{m(\frak{p})} + \frak{z}^{m(\frak{l})}}
\end{array}\right] \cong \frac{\Gamma}{\frak{p} + \frak{l}}.$$
So $$\mathcal{A}(B, B) \cong \lim_{\frak{p},\frak{l} \in \fwords{B}} \frac{\Gamma}{\frak{p} + \frak{l}}.$$ Suppose $(\gamma_{\frak{p},\frak{l}} + \frak{p} + \frak{l})_{\frak{p},\frak{l} \in \fwords{B}} \in \lim_{\frak{p},\frak{l} \in \fwords{B}} \frac{\Gamma}{\frak{p} + \frak{l}}$. Note \begin{align*}\gamma_{\frak{l},\frak{l}} - \gamma_{\frak{p}\frak{l},\frak{l}} &\in \frak{l}+\frak{l}\subseteq \frak{p}+\frak{l} \\
\gamma_{\frak{p}\frak{l},\frak{l}} - \gamma_{\frak{p},\frak{l}} &\in \frak{p}+\frak{l}\end{align*} so $$(\gamma_{\frak{p},\frak{l}} + \frak{p} + \frak{l})_{\frak{p},\frak{l} \in \fwords{B}} = (\gamma_{\frak{l},\frak{l}} + \frak{p} + \frak{l})_{\frak{p},\frak{l} \in \fwords{B}}$$ and similarly $$(\gamma_{\frak{p},\frak{l}} + \frak{p} + \frak{l})_{\frak{p},\frak{l} \in \fwords{B}} = (\gamma_{\frak{p},\frak{p}} + \frak{p} + \frak{l})_{\frak{p},\frak{l} \in \fwords{B}}.$$
We have $$(\gamma_{\frak{p},\frak{l}} + \frak{p} + \frak{l})_{\frak{p},\frak{l} \in \fwords{B}} = (\gamma_{\frak{p},\frak{p}} + \frak{p} + \frak{l})_{\frak{p},\frak{l} \in \fwords{B}} = (\gamma_{\frak{p},\frak{p}} + \frak{p})_{\frak{p} \in \fwords{B}}.(1 + \frak{p} + \frak{l})_{\frak{p},\frak{l} \in \fwords{B}}$$ and $$(\gamma_{\frak{p},\frak{l}} + \frak{p} + \frak{l})_{\frak{p},\frak{l} \in \fwords{B}} = (\gamma_{\frak{l},\frak{l}} + \frak{p} + \frak{l})_{\frak{p},\frak{l} \in \fwords{B}} = (1 + \frak{p} + \frak{l})_{\frak{p},\frak{l} \in \fwords{B}}.(\gamma_{\frak{l},\frak{l}} + \frak{l})_{\frak{l} \in \fwords{B}}$$ showing that $\lim_{\frak{p},\frak{l} \in \fwords{B}} \frac{\Gamma}{\frak{p} + \frak{l}}$ is cyclic as both a left and right $\hat{\Gamma}_B$-module.

\renewcommand\appendix{\par
\setcounter{section}{0}%
  \renewcommand{\thesection}{\Alph{section}}
}

 

\renewcommand{\theHsection}{A\arabic{section}}

 

\appendix

\section{Appendix}\label{sec:appendix}
We provide proofs here of the properties listed in Proposition \ref{prop:comp}.
\begin{proposition}
The composition in Definition \ref{def:comp} is a well-defined $(\Gamma, \Gamma)$-bimodule map.
\end{proposition}

\begin{proof}
Let $\frak{m} \in \fwords{B}$ and $\frak{l} \in \fwords{D}$. Take $\frak{n} \in \fwords{C}$ as in Lemma \ref{doublecosetiso} so that $\frac{A}{\frak{l} A + A \frak{n}} \cong (\frac{A}{\frak{l} A})(C)$ and $\frac{A}{\frak{n} A + A \frak{m}} \cong (\frac{A}{A \frak{m}})(C)$.
We have a multiplication map:
\[
\begin{tikzcd}
  A/\frak{l} A \otimes_\Gamma A/A\frak{m} \arrow[d] & \arrow[l] A \otimes_\Gamma A \arrow[d] \\
  \frac{A}{\frak{l} A + A \frak{m}} & \arrow[l] A
\end{tikzcd}
\]
The $(\frak{m},\frak{l})$th component of our proposed composition is then given by $$\mathcal{A}(C, D) \otimes_\Gamma \mathcal{A}(B, C) \rightarrow \frac{A}{\frak{l}A + A\frak{n}} \otimes_\Gamma \frac{A}{\frak{n} A + A\frak{m}} \rightarrow (\frac{A}{\frak{l} A})(C) \otimes_\Gamma (\frac{A}{A\frak{m}})(C) \rightarrow \frac{A}{\frak{l} A + A \frak{m}}$$

As noted in Lemma \ref{doublecosetiso}, if we take $\frak{n}' \in \fwords{C}$ with $\frak{n}' \subseteq \frak{n}$, then we have:

\[
\begin{tikzcd}
  & \frac{A}{\frak{l}A + A\frak{n}'} \otimes_\Gamma \frac{A}{\frak{n}' A + A\frak{m}} \arrow[dd] \arrow[dr]\\
  \mathcal{A}(C, D) \otimes_\Gamma \mathcal{A}(B, C) \arrow[ur] \arrow[dr] & & (\frac{A}{\frak{l} A})(C) \otimes_\Gamma (\frac{A}{A\frak{m}})(C)\\
  & \frac{A}{\frak{l}A + A\frak{n}} \otimes_\Gamma \frac{A}{\frak{n} A + A\frak{m}} \arrow[ur]
\end{tikzcd}
\]
so the component map is well defined.

Now suppose $\frak{m}' \in \fwords{B}$ and $\frak{l}' \in \fwords{D}$ so that $\frak{m}' \subseteq \frak{m}$ and $\frak{l}' \subseteq \frak{l}$. Take $\frak{n}' \in \fwords{C}$ with $\frak{n}' \subseteq \frak{n}$ so that $\frac{A}{\frak{l}' A + A \frak{n}'} \cong (\frac{A}{\frak{l}' A})(C)$ and $\frac{A}{\frak{n}' A + A \frak{m}'} \cong (\frac{A}{A \frak{m}'})(C)$. Then the following diagram commutes:

\[
\begin{tikzcd}
  & \frac{A}{\frak{l}'A + A\frak{n}'} \otimes_\Gamma \frac{A}{\frak{n}' A + A\frak{m}'} \arrow[dd] \arrow[r] & (\frac{A}{\frak{l}' A})(C) \otimes_\Gamma (\frac{A}{A\frak{m}'})(C) \arrow[dd] \arrow[r] & \frac{A}{\frak{l}'A + A \frak{m}'} \arrow[dd]\\
  \mathcal{A}(C, D) \otimes_\Gamma \mathcal{A}(B, C) \arrow[ur] \arrow[dr] & \\
  & \frac{A}{\frak{l}A + A\frak{n}} \otimes_\Gamma \frac{A}{\frak{n} A + A\frak{m}} \arrow[r] & (\frac{A}{\frak{l} A})(C) \otimes_\Gamma (\frac{A}{A\frak{m}})(C) \arrow[r] & \frac{A}{\frak{l}A + A \frak{m}}
\end{tikzcd}
\]
So that this gives the desired map $\mathcal{A}(C, D) \otimes_\Gamma \mathcal{A}(B, C) \rightarrow \mathcal{A}(B, D)$.
\end{proof}

\begin{proposition}
For $\gamma \in \Gamma$ and $\alpha \in \mathcal{A}(B, C)$, $\gamma.\alpha = (\gamma) \circ \alpha$ and $\alpha.\gamma = \alpha \circ (\gamma)$. Consequently, $(1) \in \mathcal{A}(B, C)$ is the identity element.
\end{proposition}

\begin{proof}
Let $\gamma \in \Gamma$, $\alpha \in \mathcal{A}(B, C)$, $\frak{m} \in \fwords{B}$ and $\frak{l} \in \fwords{C}$. Pick $\frak{n} \in \fwords{C}$ as in Lemma \ref{doublecosetiso}. We also pick $a_0$ so that:

\begin{align*}
    \alpha_{\frak{m},\frak{n}} &= a_0 + \frak{n}A + A \frak{m} & & a_0 + A \frak{m} \in (\frac{A}{A \frak{m}})(C)
\end{align*}
Since $\gamma \in \Gamma$, we already have $\gamma + \frak{l} A \in (\frac{A}{\frak{l} A})(C)$.
So $\big((\gamma) \circ \alpha \big)_{\frak{m},\frak{l}} = \gamma a_0 + \frak{l} A + A \frak{m} = \gamma .\big(a_0 + \frak{l} A + A \frak{m} \big)$, and $(\gamma) \circ \alpha = \gamma.\alpha$. The right action follows similarly.

\end{proof}

\begin{proposition}\label{associativity}
The composition $\mathcal{A}(C, D) \otimes_\Gamma \mathcal{A}(B, C) \rightarrow \mathcal{A}(B, D)$ is associative.
\end{proposition}

\begin{proof}
Let $\delta \in \mathcal{A}(D, P)$, $\beta \in \mathcal{A}(C, D)$, and $\alpha \in \mathcal{A}(B, C)$. Fix $\frak{m} \in \fwords{B}$ and $\frak{p} \in \fwords{P}$.
Pick $\frak{l} \in \fwords{D}$ so that:
\begin{align*}
    \delta_{\frak{l},\frak{p}} &= d_0 + \frak{p} A + A \frak{l} & & d_0 + \frak{p} A \in (\frac{A}{\frak{p} A})(D)\\
    (\beta \circ \alpha)_{\frak{m},\frak{l}} &= c_0 + \frak{l} A + A \frak{m} & &  c_0 + A \frak{m} \in (\frac{A}{A \frak{m}})(D)\\
    (\delta \circ (\beta \circ \alpha))_{\frak{m},\frak{p}} &= d_0c_0 + \frak{p} A + A \frak{m}
\end{align*}
Then pick $\frak{n} \in \fwords{C}$ so that:
\begin{align*}
    \beta_{\frak{n},\frak{l}} &= b_0 + \frak{l} A + A \frak{n} & & b_0 + \frak{l} A \in (\frac{A}{\frak{l} A})(C)\\
    \alpha_{\frak{m},\frak{n}} &= a_0 + \frak{n} A + A \frak{m} & & a_0 + A \frak{m} \in (\frac{A}{A \frak{m}})(C)\\
    (\beta \circ \alpha)_{\frak{m},\frak{l}} &= b_0a_0 + \frak{l} A + A \frak{m}
\end{align*}
Now pick $\frak{n}' \in \fwords{C}$ with $\frak{n}' \subseteq \frak{n}$ so that:
\begin{align*}
    (\delta \circ \beta)_{\frak{n}',\frak{p}} &= w_0 + \frak{p} A + A \frak{n}' & &  w_0 + \frak{p} A \in (\frac{A}{\frak{p} A})(C)\\
    \alpha_{\frak{m},\frak{n}'} &= x_0 + \frak{n}' A + A \frak{m} & & x_0 + A \frak{m} \in (\frac{A}{A \frak{m}})(C)\\
    ((\delta \circ \beta) \circ \alpha)_{\frak{m},\frak{p}} &= w_0x_0 + \frak{p} A + A \frak{m}
\end{align*}
Then pick $\frak{l}' \in \fwords{D}$ with $\frak{l}' \subseteq \frak{l}$ so that:
\begin{align*}
    \delta_{\frak{l}',\frak{p}} &= z_0 + \frak{p} A + A \frak{l}' & & z_0 + \frak{p} A \in (\frac{A}{\frak{p} A})(D)\\
    \beta_{\frak{n}',\frak{l}'} &= y_0 + \frak{l}' A + A \frak{n}' & & y_0 + A \frak{n}' \in (\frac{A}{A \frak{n}'})(D)\\
    (\delta \circ \beta)_{\frak{n}',\frak{p}} &= z_0y_0 + \frak{p} A + A \frak{n}'
\end{align*}
We will make use of the following list of congruences:
\begin{align*}
    z_0 - d_0 &\in \frak{p} A + A \frak{l} \\
    y_0 - b_0 &\in \frak{l} A + A \frak{n} \\
    x_0 - a_0 &\in \frak{n} A + A \frak{m} \\
    c_0 - b_0a_0 &\in \frak{l} A + A \frak{m} \\
    w_0 - z_0y_0 &\in \frak{p} A + A \frak{n}
\end{align*}
From these, we deduce the congruences:
\begin{align*}
    (d_0 - z_0)c_0 \in \frak{p} A + A \frak{m}\\
    z_0(c_0 - b_0a_0) \in \frak{p} A + A \frak{m} \\
    z_0b_0(a_0 - x_0)\in \frak{p} A + A \frak{m}\\
    z_0(b_0 - y_0)x_0 \in \frak{p} A + A \frak{m}\\
    (z_0y_0 - w_0)x_0 \in \frak{p} A + A \frak{m}
\end{align*}
Thus we have the associativity of composition:
$$(\delta \circ (\beta \circ \alpha))_{\frak{m},\frak{p}} = d_0c_0 +\frak{p} A + A \frak{m} = w_0x_0 +\frak{p} A + A \frak{m} = ((\delta \circ \beta) \circ \alpha)_{\frak{m},\frak{p}}$$ for each $\frak{m} \in \fwords{B}$ and $\frak{p} \in \fwords{P}$.
\end{proof}

\bibliographystyle{siam}
\bibliography{refs}

\end{document}